\documentclass[reqno,15pt]{amsart}
\usepackage{amsmath}
\usepackage{amssymb}
\usepackage{amscd}
\usepackage{graphics}
\usepackage{latexsym}
\usepackage{hyperref}

\theoremstyle{plain}
\newtheorem{theorem}{Theorem}[section]
\newtheorem{thm}[theorem]{Theorem}
\newtheorem{cor}[theorem]{Corollary}
\newtheorem{lem}[theorem]{Lemma}
\newtheorem{prop}[theorem]{Proposition}

\newcounter{kludge}
\newcounter{kludgeb}
\theoremstyle{definition}
\newtheorem{defn}[theorem]{Definition}

\newtheorem{rmk}[theorem]{Remark}

\newtheorem{ex}[theorem]{Example}
\newtheorem{hyp}[theorem]{Hypothesis}

\theoremstyle{remark}

\input xy
\xyoption{all}

\newcommand{\marpar}[1]{}
\newcommand{\mni}{\medskip\noindent}

\newcommand{\mbb}{\mathbb}

\newcommand{\ZZ}{\mbb{Z}}

\newcommand{\PP}{\mbb{P}}

\newcommand{\mc}{\mathcal}
\newcommand{\mcA}{\mc{A}}
\newcommand{\mcB}{\mc{B}}
\newcommand{\mcC}{\mc{C}}

\newcommand{\mcH}{\mc{H}}

\newcommand{\mcL}{\mc{L}}
\newcommand{\mcM}{\mc{M}}
\newcommand{\mcN}{\mc{N}}

\newcommand{\mcP}{\mc{P}}

\newcommand{\mcT}{\mc{T}}

\newcommand{\mcX}{\mc{X}}

\newcommand{\mf}{\mathfrak}

\newcommand{\OO}{\mc{O}}

\newcommand{\wt}{\widetilde}

\newcommand{\ol}{\overline}
\newcommand{\ul}{\underline}

\newcommand{\SP}{\text{Spec }}

\newcommand{\Pic}[2]{\text{Pic}^{#1}_{#2}}

\newcommand{\Gm}[1]{\mathbb{G}_{#1}}

\newcommand{\pr}[1]{\text{pr}_{#1}}

\newcommand{\Cc}{C}

\newcommand{\Lb}{\Lambda}

\newcommand{\sS}{\mathsf{S}}

\newcommand{\lt}{\left}
\newcommand{\rt}{\right}

\newsavebox{\sembox}
\newlength{\semwidth}
\newlength{\boxwidth}

\newsavebox{\semrbox}
\newlength{\semrwidth}
\newlength{\boxrwidth}

\title
{Intersection Sheaves for Abel Maps} 

\author[Starr]{Jason Michael Starr}
\address{Department of Mathematics \\
  Stony Brook University \\ Stony Brook, NY 11794}
\email{jstarr@math.stonybrook.edu} 

\date{\today}

\begin{document}


\begin{abstract}
Intersection sheaves, i.e., the Deligne pairings, were first introduced
by Deligne in the setting of Poincar\'{e} duality for \'{e}tale
cohomology, and later in his work on the determinant of cohomology.  
Intersection sheaves were generalized from smooth schemes to
Cohen-Macaulay schemes by Elkik, and then beyond Cohen-Macaulay
schemes by Munoz-Garcia.  To define the Abel maps arising in rational
simple connectedness on the natural parameter spaces of rational sections,
we need a variant
of the construction of Munoz-Garcia that has the good properties of
the construction using det and Div.  In addition, we prove basic
properties of the classifying stacks that arise in the definition of
Abel maps.
\end{abstract}


\maketitle



\section{Introduction} \label{sec-intro}
\marpar{sec-intro}

\mni
The Deligne pairing, or intersection sheaf, 
was introduced for families of smooth,
proper 
curves in
\cite[Expos\'{e}
XVIII]{SGA4T3} as part of the proof of Poincar\'{e} duality in \'{e}tale
cohomology.   
It was developed for more general smooth morphisms
in \cite{DeligneDet}
where Deligne also enriched the determinant of cohomology in many
ways.  In contrast to the Chow-theoretic 
pushforward of an intersection product of
divisor classes, which it closely mirrors, the
intersection sheaf respects nilpotence of the target (\emph{some}
versions of Chow theory do not), it is \emph{integral} in the sense
that there are no denominators, and it is defined for a robust class
of targets including mixed characteristic schemes.  It is well-adapted
to refined versions such as in Arakelov theory.  Finally, and this is
crucial here, via its additivity property it extends from
$\Gm{m}$-torsors, i.e., invertible sheaves, to torsors for more
general group schemes of multiplicative type. 

\mni
Intersection sheaves satisfying the axioms of the original
construction were extended to families of (not-necessarily-smooth)
Cohen-Macaulay schemes in 
\cite{Elkik}, and finally extended to
non-Cohen-Macaulay schemes in 
\cite{MunozGarcia}.
The
construction of the Abel map in \cite{dJHS} is a special case of an
intersection sheaf in relative dimension $0$.  The intersection sheaf
in relative dimension $0$ extends 
the usual norm of invertible sheaves and Cartier divisors 
for a finite, flat morphism to morphisms that are not
necessarily finite and flat.  The construction in \cite{dJHS}
in terms of det
and Div, cf. \cite{detdiv}, is essentially the original construction.
In addition
to satisfying the axioms of intersection sheaves, the det-Div 
construction gives a
formula for the contribution to the intersection sheaf from the locus
where the morphism is not finite and flat.  This is crucial in
\cite{dJHS}: that formula implies that the Abel
maps are compatible with the inductive structure on the moduli spaces
of stable sections coming from ``boundary'' correspondences between
these moduli spaces.  

\mni
In positive characteristic, the stacks of stable sections are not
Deligne-Mumford stacks.  Even in positive characteristic, the Hilbert
scheme is a projective scheme that
parameterize sections and their specializations.  
The Hilbert scheme parameterizes closed subschemes that
are not necessarily Cohen-Macaulay.  
Thus, the appropriate version of intersection
sheaves is \cite{MunozGarcia}.
However, the
formula for the contribution to the intersection sheaf from the
non-flat locus is not part of that construction.
Our main result is that the 
construction of \cite{MunozGarcia} gives the same intersection
sheaves as the det-Div construction
in our setting.  The result, Proposition \ref{prop-intersect}, 
shows that the det-Div construction of the
intersection sheaf extends to the non-Cohen-Macaulay setting and satisfies
many of the usual axioms: enough axioms 
so that it agrees with every other construction
satisfying these axioms, including the construction of \cite{MunozGarcia}.

\mni
Since the intersection sheaves are multiadditive, they naturally
extend to torsors for $\Gm{m}^\rho$ for $\rho \geq 1$, i.e., ordered
$\rho$-tuples of invertible sheaves.
Combined with
fppf descent, which we quickly review in Section \ref{sec-torsors},
this leads to an extension of intersection sheaves to
\emph{intersection torsors} for more general group schemes than
$\Gm{m}^\rho$.  
Of course
this was part of Deligne's original construction, allowing him to
define the intersection pairing on torsors for finite, flat,
commutative group schemes.  
To define Abel maps for fibrations of higher Picard rank $\rho\geq 1$,
the relevant group
schemes are
tori that are not necessarily split.  
Corollary \ref{cor-MG2} gives the extension of intersection
sheaves to intersection torsors for tori.

\mni
The strategy of the proof of existence of
sections in \cite{dJHS} and \cite{Zhu} 
is to ascend up the tower of moduli spaces of stable sections via the boundary
correspondences and to prove that ``eventually'' the fibers of the Abel map
are rationally connected.  The terms in this tower are indexed by the
degree of a torsor on a curve together with its natural partial
ordering, so degrees of torsors are also important.  Proposition
\ref{prop-degree} gives a \emph{degree map} satisfying some natural
properties. 


\section{Axioms and the Basic Construction} \label{sec-axioms}
\marpar{sec-axioms}

\mni
For every morphism of schemes $\pi:C \to T$ that is 
proper, that is fppf of pure relative dimension $d$, and 
that is cohomologically flat in degree $0$, the relative Picard
functor is an algebraic space over $T$, $\Pic{}{C/T}$.  

\begin{defn} \label{defn-intshf} \marpar{defn-intshf}
Let $\pi$ be a morphism that is proper, fppf of pure relative
dimension $d$, and cohomologically flat in degree $0$.
An \emph{intersection datum} for $\pi$ is a tuple 
$$
(n,T_0\to T,p_0:Y_0\to T_0,g_0:Y_0\to C_0, (\mcL_0,\dots,\mcL_n))
$$
of an integer $n\geq 0$, a
$T$-scheme $T_0$, a
proper, fppf morphism $p_0$ of relative dimension $\leq
d+n$ that is projective fppf locally over $T_0$, a proper, perfect
morphism $g_0$ of $T_0$-schemes, and an $(n+1)$-tuple of invertible
sheaves on $Y_0$.  

\mni
For fixed $(n,T_0,p_0,g_0)$, an \emph{isomorphism} between
$(n+1)$-tuples $(\mcL_0,\dots,\mcL_n)$ and $(\mcL'_0,\dots,\mcL'_n)$ is
an $(n+1)$-tuple of isomorphisms of invertible sheaves $\mcL_i\to
\mcL'_i$.  
For a datum as above and for a morphism of $T$-schemes,
$T_1\to T_0$, the \emph{pullback datum} is 
$$
(n,T_1,Y_0\times_{T_0} T_1 \to T_1, Y_0\times_{T_0} T_1
\xrightarrow{g_0\times \text{Id}} C_0\times_{T_0} T_1,
(\pr{Y_0}^*\mcL_0, \dots, \pr{Y_0}^*\mcL_n)).
$$
A \emph{relative intersection sheaf} for $\pi$ is an assignment to every
intersection datum of a section $I_{g_0}(\mcL_0,\dots,\mcL_n)$ over
$T_0$ of $\Pic{}{C/T}$ that satisfies the following axioms.
\begin{enumerate}
\item[(0)]  The assignment is invariant under isomorphism of
  $(n+1)$-tuples. 
\item[(i)]  The assignment is $\mf{S}_{n+1}$-invariant.
\item[(ii)] The assignment is multiadditive in the invertible sheaves,
  i.e., 
$$
I_{g_0}(\mcL_0,\dots,\mcL_{n-1},\mcA \otimes_{\OO_{Y_0}}
  \mcB) \cong
I_{g_0}(\mcL_0,\dots,\mcL_{n-1},\mcA) 
\otimes_{\OO_{C_0}}
I_{g_0}(\mcL_0,\dots,\mcL_{n-1},\mcB).
$$
\item[(iii)] The assignment is compatible with pullback, i.e., the
  intersection sheaf of the pullback datum is the pullback to $T_1$ of
  the section of $\Pic{}{C/T}$ over $T_0$.
\item[(iv)] For $n\geq 1$,
for every $\OO_{Y_0}$-module homorphism  
$s_n:\mcL_n^{\vee} \to \OO_{Y_0}$ that is regular when restricted to
every fiber of $p_0$, for the associated Cartier divisor
$\iota_0:Z_0\hookrightarrow Y_0$, $I_{g_0\circ
  \iota_0}(\iota_0^*\mcL_0,\dots,\iota_{n-1}^*\mcL_{n-1})$ equals
$I_{g_0}(\mcL_0,\dots,\mcL_n)$.
\item[(v)] If $n$ equals $0$, then $I_{g_0}(\mcL_0)$ equals
  $\text{det}(Rg_*(\mcL_0))\otimes_{\OO_{C_0}}\text{det}(Rg_*\OO_{Y_0})^\vee$.  
\end{enumerate}
\end{defn}

\begin{rmk} \label{rmk-intshf} \marpar{rmk-intshf}
By (i), axioms (ii) and (iv) for the last argument imply the
analogous axiom for every argument.  By (ii) and (iii), 
it follows that if any $\mcL_i$ is isomorphic to a pullback from
$T_0$, 
then
$I_{g_0}(\mcL_0,\dots,\mcL_n)$ is the neutral section of $\Pic{}{C/T}$
over $T_0$.  In particular, if $Y_0$ is empty, this is the neutral
section.  By (iii) and fppf descent for morphisms, it suffices
to construct $I_{g_0}(\mcL_0,\dots,\mcL_n)$ satisfying the axioms
after fppf base change of $T_0$.  So, it suffices to construct
$I_{g_0}(\mcL_0,\dots,\mcL_n)$ when $p_0$ is projective.
Since every proper morphism $p_0$ of relative dimension $\leq 1$ is
projective fppf locally over $T_0$, it is convenient to state the
definition this way.
\end{rmk}

\begin{defn} \label{defn-gamma} \marpar{defn-gamma}
For every Noetherian
scheme $Y$, denote by $K(Y)$ the Grothendieck group of
bounded, perfect complexes of $\OO_Y$-modules with its usual ring
structure and lambda operations.  For every locally constant function
$r\in
\text{Hom}_{\text{cts}}(Y,\ZZ)$, 
(continuous for the Zariski topology and the discrete topology), 
define $r[\OO_Y]$ to be the corresponding element
in $K(Y)$. 
The \emph{rank}
of a perfect complex at a point is the alternating sum of the ranks of
the terms of
any bounded complex of locally free sheaves on a neighborhood that is
locally quasi-isomorphic to the perfect complex.  This depends only on
the perfect complex, and it
is locally constant in the point and additive for
exact triangles.  Hence it defines a group homomorphism
$\epsilon:K(Y) \to \text{Hom}_{cts}(Y,\ZZ)$.  
This is a ring homomorphism; denote the kernel by $\gamma^1$.
Together with the usual lambda operations this makes $K(Y)$ into an
augmented $\lambda$-ring (technically it splits into a tensor product
of the augmented $\lambda$-rings of the connected components of $Y$).
For every integer $n$, and for every ordered
$n$-tuple $(a_0,\dots,a_n)$ of elements $a_i\in K(Y)$ with ranks
$\epsilon_i$, denote by $\langle a_0,\dots,a_n \rangle$ the element
$(a_0-\epsilon_0[\OO_Y])\otimes \dots \otimes (a_n-\epsilon_n[\OO_Y])$.
\end{defn}

\begin{rmk} \label{rmk-gamma} \marpar{rmk-gamma}
The element $\langle a_0,\dots,a_n \rangle$ is in the gamma filtration
$\gamma^{n+1}$ on $K(Y)$.  
The gamma filtration is the filtration by ideals that is smallest
among those such that every element $\langle a_0,\dots,a_n\rangle$ is
in $\gamma^{n+1}$
and that is ``strictly'' compatible with
pullbacks to all projective space bundles (so that we can use the
``splitting principle'') in the following sense.  For every projective
bundle $f:\PP(E)\to Y$ of relative dimension $r$ and with relative
Serre twisting sheaf $\OO(1)$, an element $a\in K(Y)$ is in
$\gamma^{n+1}$ if and only if $([\OO(1)]-[\OO_{\PP(E)}])^r\otimes
f^*a$ is in $\gamma^{n+r+1}$, cf. \cite[Corollary 8.10]{ManinK}.
The operation $(a_0,\dots,a_n)\mapsto
\langle a_0,\dots,a_n\rangle$ is
multiadditive for \emph{sum} in $K(Y)$.  For invertible sheaves
$(\mcL_0,\dots,\mcL_n)$, the operation $\langle \mcL_0,\dots,\mcL_n
\rangle$ is not multiadditive for tensor product of
invertible sheaves.  However, it is multiadditive for tensor products
modulo the next piece of the 
gamma filtration, $\gamma^{n+2}$.  The intersection sheaf below 
extends to an additive map $\text{det}(Rg_*(-))$
defined on $\gamma^{n+1}$.  If that map
annihilates $\gamma^{n+2}$, then the ideal sheaf
$(\mcL_0,\dots,\mcL_n) \mapsto \text{det}(Rg_*(\langle
\mcL_0,\dots,\mcL_n \rangle))$ is multiadditive for tensor products.  
It seems
difficult to prove directly that any one of the constructions of the
intersection sheaf annihilates $\gamma^{n+2}$.  
\end{rmk}

\begin{hyp} \label{hyp-intshf} \marpar{hyp-intshf}
Let $T$ be a Noetherian scheme.
Let $p:Y\to T$, $\pi:C\to T$ be morphisms of algebraic
spaces.
\end{hyp}

\begin{lem} \label{lem-310} \marpar{lem-310}
If 
$\pi$ is smooth, and
if $p$ is locally fppf, then every $T$-morphism $g:Y\to C$  
is a \emph{perfect morphism} in
the sense 
of \cite[D\'efinition III.4.1]{SGA6}.  If $\pi$ is separated and $p$
is proper,
then also $g$ is proper.
\end{lem}

\begin{proof}
Since $\pi$ is smooth, the immersion $\Gamma_g:Y\to
Y\times_T C$ is a regular immersion, hence an LCI morphism.
By
\cite[Proposition VIII.1.7]{SGA6}, $\Gamma_g$ is a perfect morphism.
If $\pi$ is separated, then $\Gamma_g$ is a closed
immersion, hence proper.
Since $p$ is locally fppf, also
$$
p\times\text{Id}_C:Y\times_T C \to C
$$
is locally fppf.  
Hence by \cite[Corollaire III.4.3.1]{SGA6}, also $p\times\text{Id}_C$ is
perfect.  
If $p$ is proper, then $p\times\text{Id}_C$ is
proper. 
Finally by \cite[Proposition III.4.5]{SGA6}, the composition
$(p\times\text{Id}_C)\circ \Gamma_g$ is perfect, i.e., $g$ is
perfect.  Also, if $p$ is proper and if $\pi$ is separated, then both
$p\times\text{Id}_C$ and $\Gamma_g$ are proper. Then the
composition $g$ is proper.
\end{proof}

\begin{hyp} \label{hyp-perfect} \marpar{hyp-perfect}
Assume that $\pi:C\to T$ is flat of constant relative dimension $d$.
Assume that $p:Y\to T$ is flat of constant relative dimension $d+n$
for an integer $n\geq 0$.  Assume that $g:C\to Y$ is a proper, perfect
$T$-morphism.  
\end{hyp}

\mni
By \cite[Proposition III.4.8]{SGA6}, for every perfect complex $E$ of
bounded amplitude on $Y$, also $Rg_*(E)$ is a perfect complex of
bounded amplitude on $C$.  Thus, by \cite{detdiv}, there is an
associated invertible sheaf $\text{det}(Rg_*(E))$ on $C$.  This is additive for
direct sum, and even for distinguished triangles of perfect complexes.  
By this
additivity property, $\text{det}(Rg_*(-))$ extends to an additive
homomorphism from the Grothendieck group of  
virtual classes of perfect complexes on $Y$ to the Picard group of $C$. 

\begin{defn} \label{defn-DP} \marpar{defn-DP}
Assuming Hypothesis \ref{hyp-perfect},
for every ordered $(n+1)$-tuple $(\mcL_0,\dots,\mcL_n)$ of invertible
sheaves on $Y$, the \emph{intersection sheaf} or 
\emph{Deligne pairing} relative to $g$, $I_g( \mcL_0, \dots, \mcL_n
)$, is the invertible sheaf on $C$ 
obtained by applying $\text{det}(Rg_*(-))$ to the virtual class
$\langle \mcL_0,\dots,\mcL_n\rangle :=
([\mcL_0]-[\OO_Y])\otimes \dots \otimes ([\mcL_n]-[\OO_Y]).$
\end{defn}

\mni
The virtual class $\langle \mcL_0,\dots,\mcL_n\rangle$ 
is basically the cup product of the K-theoretic first-Chern-classes (up
to twisting by the K-theory class of an invertible sheaf).  This is
an explicit alternating sum of K-theory 
classes of invertible $\OO_Y$-modules, and $\text{det}(Rg_*(-))$ of
this class is the alternating tensor product of the corresponding
terms.  
The
intersection sheaf  
is functorial for
the category of ordered $(n+1)$-tuples of invertible $\OO_Y$-modules
with isomorphisms as the morphisms.   
This is also $\mf{S}_{n+1}$-equivariant for permutation of the terms of
$(\mcL_0,\dots,\mcL_n)$.  If any $\mcL_i$ is isomorphic to $\OO_Y$,
then the intersection sheaf is isomorphic to $\OO_C$.  


\section{Regular Sequences} \label{sec-regseq}
\marpar{sec-regseq}

\mni
The intuition of intersection sheaves is stated most simply in terms of
regular sequences, and this is the basis of the construction for many
authors.  Whether or not it is the basis of any particular
construction, it is a key technique for proving properties of the
intersection sheaves.

\mni
Let $(s_i)_i$ be an $\OO_Y$-module homomorphism,
$$
(s_i)_{0\leq i\leq n}: \bigoplus_{i=0}^n \mcL_i^\vee \to \OO_Y.
$$

\begin{hyp} \label{hyp-Koszul} \marpar{hyp-Koszul}
Assume that $d\geq 1$, and assume that every fiber of $\pi$ satisfies
Serre's condition ${\sS}_2$ (if $d$ equals $1$, then conditions $\sS_1$ and
${\sS}_2$ are equivalent).
Assume that the restriction of $(s_i)_i$ to every fiber
of $p$ is a regular sequence.  
\end{hyp}

\begin{prop} \label{prop-Koszul} \marpar{prop-Koszul}
Under the above hypotheses, there exists 
$s=I_g(s_0,\dots,s_n)$, an $\OO_C$-module homomorphism $\OO_C\to
I_g(\mcL_0,\dots,\mcL_n)$ whose restriction to every fiber of $\pi$ is
injective.  The closed subset $g(\text{Supp}(\text{Coker}(s_i)_i))$
contains 
$\text{Supp}(\text{Coker}(s))$. 
Finally, for every virtual linear combination $\mc{H}$ of locally free
$\OO_Y$-modules  
of virtual rank $r\in \ZZ$, $\text{det}(Rg_*(\mc{H}\otimes \langle
\mcL_0,\dots,\mcL_n \rangle))$ is isomorphic to
$I_g(\mcL_0,\dots,\mcL_n)^{\otimes r}$.  
\end{prop}

\begin{proof}
Denote by $K(s_i)_i$ the Koszul complex of $(s_i)_i$.  This is the
free differential graded $\OO_Y$-algebra on $(s_i)_i:\bigoplus_i
\mcL_i^\vee\to \OO_Y$.  It is a complex of locally free sheaves
concentrated in degrees $[-n-1,0]$.  In degree $-d$, the associated
locally free sheaf is
$$
K_d(s_i)_i = \bigwedge^d_{\OO_Y}\lt(\bigoplus_{i=0}^n \mcL_i^\vee \rt).
$$
Thus, the $K$-theory class of $K(s_i)_i$ is $\mcL\otimes \langle
\mcL_0,\dots,\mcL_n \rangle$, for the invertible sheaf 
$$
\mcL =
\mcL_0\otimes \dots \otimes \mcL_n.
$$
By the hypothesis, the only nonzero homology sheaf of this complex is
in degree $0$, and that homology sheaf is $\text{Coker}(s_i)_i$.  By
hypothesis, this is flat over $T$ of relative dimension $d-1$.  

\mni
Because the fibers of $\pi$ satisfy ${\sS}_2$, every point of $C$ of depth
$\leq 0$ is a generic point of its $\pi$-fiber, and every point of $C$
of depth $\leq 1$ is either a generic point or a codimension $1$ point
of its $\pi$-fiber.  By hypothesis, after restricting to the fiber
over each point $t$ of $T$, the support of $\text{Coker}(s_i)_i$ has
dimension $\leq d-1$.  Since the $\pi$-fiber $C_t$ has dimension $d$, the
fiber of $g$ over every codimension $0$ point of $C_t$ is disjoint
from the support of $\text{Coker}(s_i)_i$.  Similarly, the fiber of
$g$ over every codimension $1$ point of $C_t$ intersects the support
of $\text{Coker}(s_i)_i$ in a zero-dimensional scheme.
Altogether, $K(s_i)_i$ satisfies the transversality hypothesis
$Q_{-1}$ relative to $g:Y\to C$, \cite[Definition, p. 50]{detdiv}; in
fact it even satisfies the hypothesis after restricting to the fiber 
over every point of $T$.  

\mni
By
\cite[Proposition 9]{detdiv}, there is an associated Div section
$$
s:\OO_C\to \text{det}(Rg_* K(s_i)_i),
$$
and for every locally free sheaf $\mcH$ of rank $r$,
$\text{det}(Rg_*(\mc{H}\otimes K(s_i)_i))$ is canonically isomorphic
to 
$\text{det}(Rg_*
K(s_i)_i)^{\otimes r}$. (Please note: the statement of the proposition
is only for $r=1$, but the proof of the proposition works for all
$r\geq 1$.)  

\mni
In particular, for $\mcH$ equal to $\mcL$,
$\text{det}(Rg_* K(s_i)_i)$ is canonically isomorphic to
$I_g(\mcL_0,\dots,\mcL_n)$.  Using additivity of
$\text{det}(Rg_*(-))$, for every virtual linear combination $\mc{H}$
of locally free $\OO_Y$-modules, $\text{det}(Rg_*(\mc{H}\otimes
\langle \mcL_0,\dots,\mcL_n \rangle))$ is isomorphic to
$I_g(\mcL_0,\dots,\mcL_n)^{\otimes r}$.    
\end{proof}

\mni
Deligne's study of the functoriality of $I_g(s_0,\dots,s_n)$ in
$(s_0,\dots,s_n)$ is used to prove important properties of
$I_g(\mcL_0,\dots,\mcL_n)$ such as additivity.  
Let $\mc{A}$ and $\mc{B}$ be invertible sheaves on $Y$.  Let 
$$
\alpha:\mc{A}^\vee\to \OO_Y, \ \ \beta:\mc{B}^\vee \to \OO_Y
$$
be $\OO_Y$-module homomorphisms.  Then $\mc{A}\otimes_{\OO_Y} \mc{B}$
is an invertible sheaf on $Y$, and there is a tensor product
$\OO_Y$-module homomorphism,
$$
\alpha\otimes \beta:\mc{A}^\vee\otimes_{\OO_Y}\mc{B}^\vee \to \OO_Y.
$$
There is a useful compatibility between the Koszul complexes
$K(\alpha)$, $K(\beta)$, $K(\alpha\otimes \beta)$, and
$K(\alpha,\beta)$.  

\mni
There are natural morphisms of differential graded $\OO_Y$-algebras,
$K(\alpha\otimes \beta)\to K(\alpha)$ and $K(\alpha\otimes\beta)\to
K(\beta)$.  
There are morphisms of complexes of $\OO_Y$-modules, 
$$
L_\alpha: \mc{A}^\vee\otimes_{\OO_Y}K(\beta) \to K(\alpha\otimes
\beta), 
$$
$$
\begin{CD}
\mc{A}^\vee\otimes_{\OO_Y}\mc{B}^\vee @> \text{Id}\otimes \beta >>
\mc{A}^\vee \\ 
@V \text{Id} VV  @VV \alpha\otimes \text{Id} V \\
\mc{A}^\vee\otimes_{\OO_Y}\mc{B}^\vee @> \alpha\otimes \beta >> \OO_Y
\end{CD},
$$
$$
L_\beta: K(\alpha)\otimes_{\OO_Y} \mc{B}^\vee \to K(\alpha\otimes
\beta), 
$$
$$
\begin{CD}
\mc{A}^\vee\otimes_{\OO_Y}\mc{B}^\vee @> \alpha\otimes \text{Id} >>
\mc{B}^\vee \\ 
@V \text{Id} VV  @VV \text{Id} \otimes \beta V \\
\mc{A}^\vee\otimes_{\OO_Y}\mc{B}^\vee @> \alpha\otimes \beta >> \OO_Y
\end{CD}.
$$
Since $\mc{A}$ and $\mc{B}$ are invertible sheaves, these are even
homomorphisms of differential graded modules over $K(\alpha\otimes
\beta)$.  
The direct sum is a morphism of complexes of $\OO_Y$-modules,
$$
L_\alpha\oplus L_\beta:\lt( \mc{A}^\vee\otimes_{\OO_Y} K(\beta) \rt) \oplus
\lt( K(\alpha)\otimes \mc{B}^\vee \rt) \to K(\alpha\otimes \beta).
$$
Define $C(\alpha,\beta)$ to be the mapping cone of $L_\alpha\oplus L_\beta$,
$$
\begin{CD}
\lt(\mc{A}^\vee\otimes_{\OO_Y} \mc{B}^\vee\rt) \oplus \lt(\mc{A}^\vee
\otimes_{\OO_Y} \mc{B}^\vee \rt) 
@> d_{C,2} >>
\mc{A}^\vee \oplus \mc{B}^\vee \oplus \lt(\mc{A}^\vee \otimes_{\OO_Y}
\mc{B}^\vee\rt) 
@> d_{C,1} >>
\OO_Y,
\end{CD}
$$
$$
d_{C,1} = \lt[ \begin{array}{ccc}  
-\alpha & -\beta & \alpha\otimes \beta
\end{array} \rt], 
$$
$$
d_{C,2} = \lt[ \begin{array}{cc}
-\text{Id}\otimes \beta &                        0 \\
                      0 & -\alpha\otimes \text{Id} \\
             -\text{Id} &               -\text{Id}
\end{array} \rt].
$$
This is a differential graded $\OO_Y$-algebra.  
There is a mapping cone short exact sequence,
\marpar{eqn-first}
\begin{equation} 
0 \to K(\alpha\otimes \beta) \xrightarrow{u}
C(\alpha,\beta) \xrightarrow{v}
\lt(\mc{A}^\vee\otimes_{\OO_Y} K(\beta)\rt)\oplus
\lt(K(\alpha)\otimes_{\OO_Y} \mc{B}^\vee \rt)[1] \to 0.
\end{equation}
\label{eqn-first} 

\mni
The morphism $u$ is a homomorphism of differential graded
$\OO_Y$-algebras, and $v$ is a homomorphism of differential graded
$K(\alpha\otimes \beta)$-modules.

\mni
There is a natural morphism of differential graded $\OO_Y$-algebras,
$K(\alpha,\beta)\to K(\text{Id})$.  
The zero map, 
$$
0:\mcA^\vee\otimes_{\OO_Y}\mcB\otimes_{\OO_Y} K(\text{Id}) \to
K(\alpha,\beta),
$$
is a homomorphism of differential graded $K(\alpha,\beta)$-modules.
The mapping cone $C'(\alpha,\beta)$ is just
$\lt(\mcA^\vee\otimes_{\OO_Y}
\mcB^\vee\otimes_{\OO_Y}K(\text{Id})\rt)[1]\oplus K(\alpha,\beta),$
$$
\begin{CD}
\lt(\mc{A}^\vee\otimes_{\OO_Y} \mc{B}^\vee\rt) \oplus \lt(\mc{A}^\vee
\otimes_{\OO_Y} \mc{B}^\vee \rt) 
@> d_{C',2} >>
\lt(\mc{A}^\vee \otimes_{\OO_Y}
\mc{B}^\vee\rt) \oplus \mc{A}^\vee \oplus \mc{B}^\vee
@> d_{C',1} >>
\OO_Y,
\end{CD}
$$
$$
d_{C',1} = \lt[ \begin{array}{ccc}
0 & \alpha & \beta
\end{array} \rt],
$$
$$
d_{C',2} = \lt[ \begin{array}{cc}
 \text{Id} &                       0 \\
         0 & -\text{Id}\otimes \beta \\
         0 & \alpha\otimes \text{Id}
\end{array} \rt].
$$
\mni
This is a differential graded $\OO_Y$-algebra.
There is a mapping cone short exact sequence,
\marpar{eqn-second}
\begin{equation} 
0 \to K(\alpha,\beta) \xrightarrow{u'} C'(\alpha,\beta)
\xrightarrow{v'} \mc{A}^\vee\otimes_{\OO_Y}\mc{B}^\vee\otimes_{\OO_Y}
K(\text{Id}) \to 0.
\end{equation}
\label{eqn-second} 

\mni
The morphism $u'$ is a homomorphism of differential graded
$\OO_Y$-algebras, and $v'$ is a homomorphism of differential graded
$K(\alpha,\beta)$-modules.  


\begin{lem} \label{lem-compat} \marpar{lem-compat}
The complexes of $\OO_Y$-modules, $C(\alpha,\beta)$ and
$C'(\alpha,\beta)$, are isomorphic.
\end{lem}

\begin{proof}
Consider the diagram
$$
\begin{CD}
\lt(\mc{A}^\vee\otimes_{\OO_Y} \mc{B}^\vee\rt) \oplus \lt(\mc{A}^\vee
\otimes_{\OO_Y} \mc{B}^\vee \rt) 
@> d_{C,2} >>
\mc{A}^\vee \oplus \mc{B}^\vee \oplus \lt(\mc{A}^\vee \otimes_{\OO_Y}
\mc{B}^\vee\rt) 
@> d_{C,1} >>
\OO_Y \\
@V \phi_2 VV @V \phi_1 VV @VV \text{Id} V \\
\lt(\mc{A}^\vee\otimes_{\OO_Y} \mc{B}^\vee\rt) \oplus \lt(\mc{A}^\vee
\otimes_{\OO_Y} \mc{B}^\vee \rt) 
@> d_{C',2} >>
\lt(\mc{A}^\vee \otimes_{\OO_Y}
\mc{B}^\vee\rt) \oplus \mc{A}^\vee \oplus \mc{B}^\vee
@> d_{C',1} >>
\OO_Y,
\end{CD}
$$
$$
\phi_1 = \lt[ \begin{array}{ccc}
         0 &           0 &              \text{Id} \\
-\text{Id} &           0 & \text{Id}\otimes \beta \\
         0 & - \text{Id} &                      0
\end{array} \rt],
$$
$$
\phi_2 = \lt[ \begin{array}{cc}
-\text{Id} & -\text{Id} \\
         0 &  \text{Id}
\end{array} \rt].
$$
It is straightforward to check that $\phi$ is a morphism of complexes.
Also, each of $\phi_1$ and $\phi_2$ is an isomorphism.  Thus, $\phi$
is an isomorphism of complexes of $\OO_Y$-modules.
\end{proof}

\mni
In what follows, all we need is existence of this isomorphism; we need
no other compatibilities.  That is fortunate:
the isomorphism above is not unique, and $\phi$ is not an isomorphism
of differential graded $\OO_Y$-algebras.  
Nonetheless, it is a lifting to perfect differential graded
$\OO_Y$-algebras of an elementary homomorphism of $\OO_Y$-algebras.
Since
$\mc{A}^\vee\otimes_{\OO_Y}\mc{B}^\vee\otimes_{\OO_Y}K(\text{Id})$ is
an acyclic complex, $u'$ is a quasi-isomorphism.  Thus, in the derived
category of complexes of $\OO_Y$-modules, there is a
distinguished triangle,
$$
\lt( \mc{A}^\vee \otimes_{\OO_Y} K(\beta)  \rt) \oplus 
\lt( K(\alpha) \otimes_{\OO_Y} \mc{B}^\vee \rt) 
\xrightarrow{-L_\alpha\oplus L_\beta}
K(\alpha\otimes \beta)
\xrightarrow{(u')^{-1}\circ\phi\circ u}
K(\alpha,\beta) 
$$
$$
\ \ \ \ \ \ \ \ \ \ \ \ \ \ \ \ \ \ \ \ \ \ \ \ \ {}
\xrightarrow{v\circ \phi^{-1} \circ u'}
\lt( \mc{A}^\vee \otimes_{\OO_Y} K(\beta)  \rt) \oplus 
\lt( K(\alpha) \otimes_{\OO_Y} \mc{B}^\vee \rt)[1] 
$$
When $(\alpha,\beta)$ is a regular sequence, this distinguished
triangle 
is a lifting to complexes of 
locally free sheaves of the
short exact sequence of $\OO_Y$-modules,
$$
0 \to 
\lt( \mc{A}^\vee \otimes_{\OO_Y}\lt( \OO_Y/\beta\rt) \rt) \oplus
\lt(\lt( \OO_Y/\alpha\rt) \otimes_{\OO_Y} \mc{B}^\vee \rt) \to
\OO_Y/(\alpha\otimes \beta) \to
\OO_Y/(\alpha \oplus \beta) \to 
0.
$$

\begin{hyp} \label{hyp-Koszul2} \marpar{hyp-Koszul2}
Assume that $d\geq 1$, and assume that every fiber of $\pi$ satisfies
Serre's condition ${\sS}_2$ (if $d$ equals $1$, then conditions $\sS_1$ and
${\sS}_2$ are equivalent).
Assume that the restriction of $(s_0,\dots,s_{n-1},\alpha,\beta)$ 
to every fiber
of $p$ is a regular sequence.  
\end{hyp}

\begin{prop} \label{prop-cones} \marpar{prop-cones}
Under the above hypotheses, there is an equality of effective
Cartier divisors, $I_g(s_0,\dots,s_{n-1},\alpha\otimes \beta) =
I_g(s_0,\dots,s_{n-1},\alpha)\otimes I_g(s_0,\dots,s_{n-1},\beta)$.
In particular, there is an isomorphism of intersection sheaves,
$I_g(\mcL_0,\dots, \mcL_{n-1},\mc{A}\otimes_{\OO_Y}\mc{B}) \cong
I_g(\mcL_0,\dots, \mcL_{n-1},\mc{A})
\otimes_{\OO_C}
I_g(\mcL_0,\dots, \mcL_{n-1},\mc{A}).$
\end{prop}

\begin{proof}
Because the sequence $(s_0,\dots,s_{n-1},\alpha,\beta)$ is regular on
every fiber of $p$, so are the subsequences
$(s_0,\dots,s_{n-1},\alpha)$ and $(s_0,\dots,s_{n-1},\beta)$.  Also,
for every fiber of $p$, since both $\ol{\alpha}$ and $\ol{\beta}$ are
regular on the quotient by $(s_0,\dots,s_{n-1})$, also
$\ol{\alpha\otimes \beta}$ is regular on the quotient by
$(s_0,\dots,s_{n-1})$.  Thus, also the sequence
$(s_0,\dots,s_{n-1},\alpha\otimes \beta)$ is regular on every fiber of
$p$.  Thus all of the Cartier divisors above are defined.

\mni
Tensoring the Koszul complex $K(s_0,\dots,s_{n-1})$ with the short
exact sequence in Equation \ref{eqn-second} gives a mapping complex
short exact sequence,
$$
0 \to K(s_0,\dots,s_{n-1},\alpha,\beta) \xrightarrow{\text{Id}\otimes
  u'} 
K(s_0,\dots,s_{n-1})\otimes_{\OO_Y} C'(\alpha,\beta)
\xrightarrow{\text{Id}\otimes v'} 
$$
$$
\mc{A}^\vee\otimes_{\OO_Y}\mc{B}^\vee\otimes_{\OO_Y}
K(s_0,\dots,s_{n-1},\text{Id}) \to 0.
$$
The derived functor $Rg_*$ preserves mapping cones.  Thus there is
a mapping cone short exact sequence,
$$
0 \to Rg_* K(s_0,\dots,s_{n-1},\alpha,\beta) \xrightarrow{Rg_*(\text{Id}\otimes
  u')} 
Rg_*(K(s_0,\dots,s_{n-1})\otimes_{\OO_Y} C'(\alpha,\beta))
$$
$$
\xrightarrow{Rg_*(\text{Id}\otimes v')} 
Rg_*(\mc{A}^\vee\otimes_{\OO_Y}\mc{B}^\vee\otimes_{\OO_Y}
K(s_0,\dots,s_{n-1},\text{Id})) \to 0.
$$
The third complex is acyclic.
By hypothesis, $K(s_0,\dots,s_{n-1},\alpha,\beta)$
satisfies hypothesis $Q_{-2}$ relative to $g$, \cite[Definition,
p. 50]{detdiv}.  In other words, the open complement $U$ of
$g( \text{Supp} ( \text{Coker} ( s_0 , \dots , s_{n-1} , \alpha , \beta)))$
contains all depth $0$ and depth $1$ points of $C$, and
$Rg_*(K(s_0,\dots,s_{n-1}))$ is acyclic on $U$.  By \cite[Proposition
9]{detdiv}, there is a Div Cartier divisor of
$Rg_*(K(s_0,\dots,s_{n-1},\alpha,\beta))$, and this Cartier divisor is
acyclic on $U$.  By Krull's Hauptidealsatz, every minimal prime over a
principal ideal (assuming that there are any such primes) 
has height $0$ or $1$.  Thus, the Div Cartier divisor is trivial,
$\OO_C\xrightarrow{\text{Id}}\OO_C$.  Finally, by \cite[Theorem
3]{detdiv}, also $Rg_*$ of the middle complex is \emph{good}, and the
associated Div Cartier divisor is trivial.

\mni
Tensoring the Koszul complex $K(s_0,\dots,s_{n-1})$ with the
mapping cone short
exact sequence in Equation \ref{eqn-first} gives another mapping cone 
short exact
sequence,
$$
0 \to K(s_0,\dots,s_{n-1},\alpha\otimes \beta)
\xrightarrow{\text{Id}\otimes u_{L_\alpha\oplus L_\beta}}
K(s_0,\dots,s_{n-1})\otimes C(\alpha,\beta) 
\xrightarrow{\text{Id}\otimes v_{L_\alpha\oplus L_\beta}}
$$
$$
\lt(\mc{A}^\vee\otimes_{\OO_Y} K(s_0,\dots,s_{n-1},\beta)\rt)\oplus
\lt(K(s_0,\dots,s_{n-1},\alpha)\otimes_{\OO_Y} K(\alpha) \rt)[1] \to 0.
$$
The derived functor $Rg_*$ preserves mapping cones.  Thus there is
a mapping cone short exact sequence,
$$
0 \to Rg_*K(s_0,\dots,s_{n-1},\alpha\otimes \beta)
\xrightarrow{Rg_*(\text{Id}\otimes u_{L_\alpha\oplus L_\beta})}
Rg_*\lt(K(s_0,\dots,s_{n-1})\otimes C(\alpha,\beta)\rt) 
$$
$$
\xrightarrow{Rg_*(\text{Id}\otimes v_{L_\alpha\oplus L_\beta})}
Rg_*\lt(\lt(\mc{A}^\vee\otimes_{\OO_Y} K(s_0,\dots,s_{n-1},\beta)\rt)\oplus
\lt(K(s_0,\dots,s_{n-1},\alpha)\otimes_{\OO_Y} K(\alpha) \rt)\rt)[1] \to 0.
$$
The first and third complexes on $Y$ satisfy hypothesis $Q_{-1}$ relative to
$g$.  Thus, the first and third complexes on $C$ are \emph{good},
\cite[p. 47]{detdiv}.  By \cite[Theorem 3(ii)]{detdiv}, also the
middle complex on $C$ is good, and the ``sum'' of the Div Cartier
divisors of the first and third complex equals the Div Cartier divisor
of the middle complex.  By Lemma \ref{lem-compat}, the middle complex
above is isomorphic to the middle complex from the previous
paragraph.  That complex had trivial Div Cartier divisor.  Thus, the
middle complex above has trivial Div Cartier divisor.
Therefore, the Div Cartier
divisor of the first complex is the inverse Cartier divisor of the Div
Cartier divisor of the third complex.  Combined with Proposition
\ref{prop-Koszul}, this precisely gives 
$$
I_g(s_0,\dots,s_{n-1},\alpha\otimes \beta) =
I_g(s_0,\dots,s_{n-1},\beta)\otimes I_g(s_0,\dots,s_{n-1},\alpha).
$$
\end{proof}

\mni
This is the basic additivity of intersection sheaves under a
regularity hypothesis.  Via the
$\mf{S}_{n+1}$-equivariance and usual methods of multilinear algebra,
this implies other properties.  The following property is helpful in
proving additivity with no regularity hypothesis.

\mni  
For $i=0,\dots,n-1$, let $\mcL_i'$ and $\mcL_i''$ be
invertible $\OO_Y$-modules with $\OO_Y$-module homomorphisms,
$$
s'_i:(\mcL'_i)^\vee \to \OO_Y \ \ \ s''_i:(\mcL''_i)^\vee \to \OO_Y.
$$
Let $\mcA'$, $\mcA''$, $\mcB'$, and $\mcB''$ be invertible
$\OO_Y$-modules with $\OO_Y$-module homomorphisms,
$$
\alpha':(\mcA')^\vee \to \OO_Y \ \ \ \alpha'':(\mcA'')^\vee \to \OO_Y,
$$
$$
\beta':(\mcB')^\vee \to \OO_Y \ \ \ \beta'':(\mcA'')^\vee \to \OO_Y.
$$
Define $\mcL'_n = \mcA'\otimes_{\OO_Y} \mcB'$, resp. $\mcL''_n =
\mcA''\otimes_{\OO_Y} \mcB''$.  Define $s_n' = \alpha'\otimes \beta'$,
resp. $s_n''=\alpha''\otimes \beta''$.  

\mni
Denote $J=\{0,\dots,n-1,n\}$, a set with $n+1$ elements.
For every partition $J=J'\sqcup J''$, one of $J'$ or $J''$ contains
$n$.  Define $\mcL_{J',J''}$, resp. $\mcL_{J',J'',A}$,
$\mcL_{J',J'',B}$, 
to be the length-$(n+1)$ sequence of
invertible sheaves
$(\mcL_0,\dots,\mcL_{n-1},\mcL_n)$ where for $i=0,\dots,n-1$, 
$\mcL_i$ equals $\mcL'_i$ or $\mcL''_i$ depending on whether 
$i\in J'$ or $i\in J''$, and where
$\mcL_n$ equals $\mcL'_n$, resp. $\mcA'_n$, $\mcB'_n$, or
$\mcL''_n$, resp. $\mcA''_n$, $\mcB''_n$, depending on whether $n\in
J'$ or $n\in J''$.  Similarly, define $\mcL_{J',J'',A,B}$ to be the
length-$(n+1)$ sequence of invertible sheaves
$(\mcL_0,\dots,\mcL_{n-1},\mcA,\mcB)$ as above, where $\mcA$,
resp. $\mcB$, equals $\mcA'$ or $\mcA''$, resp. $\mcB'$ or $\mcB''$,
depending on whether $n\in J'$ or $n\in J''$.  For each of these
sequences, there is a corresponding sequence $s_{J',J''}$,
resp. $s_{J',J'',A}$, $s_{J',J'',B}$, $s_{J',J'',A,B}$ of
$\OO_Y$-module homomorphisms $s_j:\mcL_j^\vee\to \OO_Y$,
resp. $\alpha:\mcA^\vee \to \OO_Y$, $\beta:\mcB^\vee \to \OO_Y$, using
the $\OO_Y$-module homomorphisms from the previous paragraph. Finally,
let $(r_{J',J''})_{J',J''}$ be a sequence of integers $r_{J',J''}$
indexed by all partitions $(J',J'')$ of $J$.

\begin{hyp} \label{hyp-Koszul3} \marpar{hyp-Koszul3}
Assume that $d\geq 1$, and assume that every fiber of $\pi$ satisfies
Serre's condition ${\sS}_2$ (if $d$ equals $1$, then conditions $\sS_1$ and
${\sS}_2$ are equivalent).  Assume that for every partition $(J',J'')$ of
$J=\{0,\dots,n\}$, the restriction of the sequence $s_{J',J'',A,B}$ to
every fiber of $p$ is a regular sequence.
\end{hyp}

\begin{cor} \label{cor-cones} \marpar{cor-cones}
Under the above hypotheses,
there is an equality of Cartier divisors
(written additively),
$$
\sum_{J',J''} r_{J,J'}\text{Div}(I_g(s_{J',J''})) =
\sum_{J',J''}r_{J',J''}(\text{Div}(I_g(s_{J',J'',A})) + \text{Div}(I_g(s_{J',J'',B}))).
$$
In particular, there is an isomorphism of intersection sheaves,
$$
\bigotimes_{J',J''} I_g(\mcL_{J',J''})^{\otimes r_{J',J''}} \cong
\bigotimes_{J',J''} \lt( I_g(\mcL_{J',J'',A})\otimes_{\OO_C}
I_g(\mcL_{J',J'',B}) \rt)^{\otimes r_{J',J''}}.
$$
\end{cor}

\begin{proof}
It suffices to prove for every partition $(J',J'')$ that the Cartier divisor
$I_g(s_{J',J''})$ equals $I_g(s_{J',J'',A})\otimes I_g(s_{J',J'',B})$,
using multiplicative notation.  This follows from Proposition
\ref{prop-cones}.  
\end{proof}






\section{Properties} \label{sec-properties}
\marpar{sec-properties}

\mni
There are a few straightforward properties of the intersection sheaf.
For every morphism
$a:C_0\to C$, 
denote by $Y_0$ the fiber product $Y\times_C C_0$, and
denote by $g\times\text{Id}_{C_0}$ the projection $Y\times_C C_0 \to
C_0$.  

\begin{lem} \label{lem-C0} \marpar{lem-C0}
Assuming Hypothesis \ref{hyp-perfect}, assuming that $a$ is fppf of
pure relative dimension $e$, then $g\times \text{Id}_{C_0}$ is proper
and perfect.
For every ordered $(n+1)$-tuple $(\mcL_0,\dots,\mcL_n)$ of invertible
$\OO_Y$-modules, there is an isomorphism of invertible sheaves on
$C_0$, 
$$
I_g^a( \mcL_0, \dots, \mcL_n ): 
a^*I_g( \mcL_0 , \dots , \mcL_n ) \to 
I_{g\times \text{Id}_{C_0}}
( \text{pr}_Y^*\mcL_0 , \dots , \text{pr}_Y^* \mcL_n ).
$$
This isomorphism is natural in $(\mcL_0,\dots,\mcL_n)$ and in $a$.  
\end{lem}

\begin{proof}
The morphism $g\times\text{Id}_{C_0}$ is perfect by \cite[Corollaire
III.4.7.1]{SGA6}.  It is proper since it is a base change of a proper
morphism.  
Pullback under $\text{pr}_Y^*$ is a ring homomorphism from the
$K$-ring of virtual perfect complexes on $Y$ to the $K$-ring of
virtual perfect complexes on $Y\times_C C_0$.  Similarly,
$\text{det}(Rg_*(-))$ is compatible with $a$, \cite[p. 46]{detdiv}.
The lemma follows from these compatibilities.
\end{proof}

\mni
Similarly, for a morphism of Noetherian schemes $b:T_0\to
T$, denote $C_0=C\times_T T_0$, denote $Y_0=Y\times_T T_0$, and denote $g\times
\text{Id}_{T_0}:Y_0\to C_0$ the projection.  The same proof proves the
following.  

\begin{lem} \label{lem-T0} \marpar{lem-T0}
Assuming Hypothesis \ref{hyp-perfect},
the morphism $g\times\text{Id}_{T_0}$ is proper and perfect.
For every ordered $(n+1)$-tuple $(\mcL_0,\dots,\mcL_n)$ of invertible
$\OO_Y$-modules, there is an isomorphism of invertible sheaves on
$C_0$, 
$$
I_g^b( \mcL_0 , \dots , \mcL_n ): 
\text{pr}_C^* I_g( \mcL_0,\dots,\mcL_n ) \to 
I_{g\times \text{Id}_{T_0}}
( \text{pr}_Y^*\mcL_0,\dots, \text{pr}_Y^* \mcL_n ). 
$$
This isomorphism is natural in $(\mcL_0,\dots,\mcL_n)$ and in $b$.  
\end{lem}

\begin{hyp} \label{hyp-divisor} \marpar{hyp-divisor}
Let $p:Y\to T$ be an fppf morphism.
Let $(Z_i \hookrightarrow Y)_i$ be a nonempty finite set of 
closed subschemes that
are each $T$-flat
of constant relative dimension $d_i\geq 0$.  Let $\mcL$ be an invertible
sheaf on $Y$.  Let $\OO(1)$ be a $p$-relatively ample invertible sheaf
on $Y$.  
\end{hyp}

\begin{lem} \label{lem-divisor} \marpar{lem-divisor}
Under the above hypotheses,
there exists an integer $\wt{m}$ such that for every $m\geq \wt{m}$, 
after base change from $T$ to a
Zariski cover,
there exists a
homomorphism of coherent sheaves
$s:\OO(-m) \mcL$ such that $s$, resp. $s|_{Z_i}$, 
is injective after restriction to every fiber of $Y\to T$, 
resp. after restriction to every fiber of $Z_i\to T$.
Thus, the support of $\text{Coker}(s)$,
resp. $\text{Coker}(s|_{Z_i})$, is a $T$-flat Cartier divisor in $Y$,
resp. in $Z_i$. 
\end{lem}

\begin{proof}
By semicontinuity, there exists $\wt{m}$ such that 
for all $m\geq \wt{m}$, $p_*\mcL(m)$ surjects onto the sections of
$\mcL(m)$ on each fiber of $p$.  Thus,
every homomorphism $s$ defined on a
fiber of $p$ extends to a Zariski neighborhood of that point.
For a homomorphism $s$, resp. $s|_{Z_i}$, if the restriction to a
fiber of $p$ is 
injective, then $s$, resp. $s|_{Z_i}$, 
is injective with flat cokernel on a Zariski open
neighborhood of that fiber, 
\cite[Theorem 23.7]{Ma}.
Thus, to construct a Zariski neighborhood of a fiber of $p$
and $s$ as above, 
it suffices to prove the result
for the fiber.  Thus, assume that $T$ is $\SP k$ for a field $k$. 

\mni
By primary decomposition, there are finitely many associated points of
$Y$, resp. of the finitely many closed subschemes $Z_i$.  
By ampleness, up to increasing $\wt{m}$,
for every $m\geq \wt{m}$,
there exists $s:\OO(-m)\to \mcL$ that is an isomorphism
on the stalks at each of the finitely many associated points.  Since
the set of zero divisors in a Noetherian ring is precisely the union
of the associated primes, $s$ is injective, resp. each $s|_{Z_i}$ is injective.
\end{proof}

\begin{hyp} \label{hyp-divisor2} \marpar{hyp-divisor2}
Let $p:Y\to T$ be an fppf morphism.
Let $(Z_i \hookrightarrow Y)_i$ be a nonempty finite set of 
closed subschemes that
are each $T$-flat
of constant relative dimension $d_i\geq 0$.  
Let $\OO(1)$ be a $p$-relatively ample invertible sheaf
on $Y$.  
Let $e\geq 1$ be an integer, and let $J$ be $\{0,\dots,e\}$.
Let $(\mcL_j)_{j\in J}$ be a finite collection of
$e$ invertible sheaves on $Y$.
\end{hyp}

\begin{lem} \label{lem-divisor2} \marpar{lem-divisor2}
Under the above hypotheses,
for every integer $\wt{m}$, 
after base change from $T$ to a Zariski cover of $T$, there exists
a sequence $(m_j)_{j\in J}$ of integers $m_j\geq \wt{m}$ and there exists
a collection $(s_j)_{j\in J}$ of homomorphisms of coherent sheaves
$s_j:\OO(-m_j) \to \mcL_j$ such that for every $i$ and for every finite
subset $J'\subset J$ of size $e_i\leq d_i+1$, the sequence
$(s_j|_{Z_i})_{j\in J'}$ is regular on every fiber of $Z_i\to T$. 
Thus the closed
subscheme $Z_{i,J'}$
of $Z_i$ cut out by this regular sequence 
is $T$-flat of constant relative dimension $d_i-e_i$; $Z_i$ is empty
when $e_i$ equals $d_i+1$.  
Moreover, the
set
$\mf{M}\subset \ZZ_{\geq 0}^{n+1}$ of sequences $(m_j)_{j\in J}$ for
which there exists such a datum has the following property: for every
$(m_0,\dots,m_n) \in \mf{M}$, 
for every $r=1,\dots,n$, there exists an
integer $\wt{m}_{r}$ such that for every $m\geq \wt{m}_{r}$, there exists
$(m'_{r+1},\dots,m'_n)\in \ZZ_{\geq 0}^{n-r}$ with
$(m_0,\dots,m_{r-1},m,m'_{r+1},\dots,m'_n)$ in $\mf{M}$.
\end{lem}

\begin{proof}
This is proved by induction on $e$.  When $e$ equals
$0$, 
this follows from Lemma \ref{lem-divisor}.  Thus, by way of
induction, assume that $e>0$, and assume the result
is proved for smaller values of $e$.  Partition $J$ as $\{e\}\sqcup
J_{e}$.  By the induction
hypothesis, after base change to a Zariski cover of $T$, 
there exists 
$(m_j)_{j\in J_{e}}$ and there
exists $(s_j)_{j\in J_{e}}$ such that for
every $i$ and for every subset $J'\subset J_{e}$ of size $e_i\leq
d_i$, the sequence $(s_j|_{Z_i})_{j\in J'}$ is regular on all fibers
of $Z_i\to T$.  Thus the closed subscheme $Z_{i,J'}$ of $Z_i$ cut out
by this regular sequence is $T$-flat of relative dimension $d_i-e_i$.

\mni
For every $(m_j)_{j\in J_{e}}$ and $(s_j)_{j\in J_{e}}$ as above,
consider the collection
$(Z_{i,J'})_{(i,J')}$ of $T$-flat closed subschemes
of $Y$ where for every $i$, $J'$ varies over 
all subsets
$J'\subset J_{e-1}$ of size $e_i \leq d_i$.  When $J'$ is the empty set,
interpret $Z_{i,\emptyset}$ as $Z_i$.

\mni
By Lemma \ref{lem-divisor},
there exists an integer $\wt{m}_{e} \geq \wt{m}$ such that for every integer
$m_e\geq \wt{m}_{e}$, after replacing $T$ by a Zariski cover once more, there
exists $s_e:\OO(-m_e)\to \mcL_e$ that is regular after restriction to
every fiber of $Z_{i,J'}\to T$.  Thus, since the
restriction to each fiber of $Z_i\to T$ of $(s_j)_{j\in J'}$ is
regular, and also $s_e$ is regular on each fiber of $Z_{i,J'}\to T$,
the entire sequence $(s_j)_{j\in J'}\cup (s_e)$ is regular on each
fiber of $Z_i\to T$.  The result follows by induction on $e$.

\mni
Moreover, since for every $(m_0,\dots,m_{e-1})$ as above, for every
$m\geq \wt{m}_e$, the element $(m_0,\dots,m_{e-1},m)$ is in $\mf{M}$,
also the observation about $\mf{M}$ follows by induction on $e$.  
\end{proof}
  
\mni
A first application of this is in the special case that every $\mcL_j$
is $\OO_Y$. 

\begin{lem} \label{lem-divisor3} \marpar{lem-divisor3}
In Lemma \ref{lem-divisor2}, if every $\mcL_j$ is $\OO_Y$, then for
every integer $m_0$, after base change from $T$ to a Zariski cover of
$T$, there exists $(m_j)_{j\in J}$ and $(s_j)_{j\in J}$ as in the
lemma and satisfying the extra hypothesis that $m_1=\dots=m_e$.
\end{lem}

\begin{proof}
By Lemma \ref{lem-divisor2}, there exists some sequence of integers
$(n_j)$, not necessarily all equal, and there exists a sequence
$(t_j)_{j\in J}$ of homomorphisms $t_j:\OO(-n_j)\to \OO_Y$ as in the
lemma.  Now let $m$ be the least common multiple of all $n_j$, i.e.,
for each $j$ there exists a positive integer $r_j$ such that
$m=n_jr_j$.  Set $s_j$ equal to $t_j^{r_j}$.  Then every $s_j$ is a
homomorphism $\OO(-m)\to \OO_Y$.  Since $(t_j)_{j\in J'}$ is regular, 
$(s_j)_{j\in J'}$ is also regular, \cite[Theorem 16.1]{Ma}. 
\end{proof}


\section{The Main Result} \label{sec-main}
\marpar{sec-main}

\begin{defn} \label{defn-SF} \marpar{defn-SF}
A \emph{Stein factorization} of $\pi$ is a pair of
finitely presented morphisms,
$$
C\xrightarrow{\pi'} T' \xrightarrow{\rho} T,
$$
such that $\rho\circ \pi'$ equals $\pi$, such that $\rho$ is
quasi-finite, and such that the natural homomorphism of $\OO_{T'}$-algebras,
$$
\OO_{T'}\to \pi'_*\OO_C,
$$ 
is an isomorphism.  
\end{defn}

\mni
If $\pi$ is proper, then there exists a Stein factorization of $\pi$.   
Assuming that a Stein factorization
exists, 
for every invertible sheaf $\mcL$
on $T'$, 
the natural homomorphism of $\OO_{T'}$-modules, 
$$
\mcL\to \pi'_*(\pi')^*\mcL,
$$
is an isomorphism.  Moreover, Stein factorizations are compatible with
fppf base change (they are compatible with arbitrary base change if
$\pi$ is proper and cohomologically flat in degree $0$).

\begin{hyp} \label{hyp-intersect} \marpar{hyp-intersect}
Assume that $\pi:C\to T$ is flat of constant relative dimension $d\geq
1$.  Assume that every fiber of $p$ satisfies Serre's condition ${\sS}_2$
(if $d$ equals $1$, then conditions $\sS_1$ and
${\sS}_2$ are equivalent).
Assume that $p:Y\to T$ is flat of constant relative dimension $d+n$
for an integer $n\geq 0$.  Assume that $g:C\to Y$ is a proper, perfect
$T$-morphism.  
\end{hyp}

\begin{prop} \label{prop-intersect} \marpar{prop-intersect}
As above, let $T$ be a Noetherian scheme. 
Let $\pi:C\to T$ be a flat morphism of constant relative dimension
$d\geq 1$, and assume that all fibers are ${\sS}_2$.  Let $p:Y\to T$ 
be a finite type,
flat morphism of constant relative dimension $d+n$ for $n\geq 0$.  Let
$g:Y\to C$ be a proper, perfect $T$-morphism.  Assume that there exists an
invertible sheaf $\OO(1)$ on $Y$ that is $p$-ample.  All of the
following hold after base change of $T$ by a Zariski cover, setting
$T'=T$; resp. if there exists a Stein factorization of $\pi$, the
following hold without base change for $T'$ as in the Stein factorization.
\begin{enumerate}
\item[(i)] 
For every $(n+2)$-tuple of invertible sheaves,
  $(\mcL_0,\dots,\mcL_{n-1},\mcA,\mcB)$, there exists an
  invertible sheaf $I_{g,\pi}(\mcL_0,\dots,\mcL_{n-1},\mcA,\mcB)$
  on $T'$ and an isomorphism of $\OO_C$-modules,
  $$
I_g(\mcL_0,\dots,\mcL_{n-1},\mcA)\otimes
  I_g(\mcL_0,\dots,\mcL_{n-1},\mcB) \cong
$$
$$
  I_g(\mcL_0,\dots,\mcL_{n-1},\mcA \otimes
  \mcB)\otimes_{\OO_{T'}}I_{g,\pi}(\mcL_0,\dots,\mcL_{n-1},\mcA,\mcB).
$$
\item[(ii)]
For every virtual perfect complex $\mc{H}$ on $Y$ of virtual rank $r$,
there exists an invertible sheaf $I_g(\mcL_0,\dots,\mcL_n,\mcH)$ on
$T'$ and an isomorphism of $\OO_C$-modules,
$$
\text{det}(Rg_*(\mc{H}\otimes \langle \mcL_0,\dots,\mcL_n \rangle))
\cong I_g(\mcL_0,\dots,\mcL_n)^{\otimes
  r}\otimes_{\OO_{T'}}I_g(\mcL_0,\dots,\mcL_n,\mcH).
$$
\item[(iii)]
For every $\OO_Y$-module homomorphism $s_n:\mcL_n^\vee \to \OO_Y$
whose restriction to every fiber of $p$ is injective, for the closed
subscheme $\iota:Z\hookrightarrow Y$ that is the effective Cartier
divisor of $s$, there exists an invertible sheaf
$I_g(\mcL_0,\dots,\mcL_n,s_n)$ on $T'$ and an isomorphism of
$\OO_C$-modules,
$$
I_{g\circ \iota}(\iota^*\mcL_0,\dots,\iota^*\mcL_{n-1}) \cong
I_g(\mcL_0,\dots,\mcL_{n-1},\mcL_n) \otimes_{\OO_{T'}}
I_g(\mcL_0,\dots,\mcL_n,s_n). 
$$
\end{enumerate} 
\end{prop}

\begin{proof}
For a Stein factorization,
since $\pi'_*(\pi')^*\mcL$ equals $\mcL$ for every invertible sheaf on
$T'$, the invertible sheaves
$I_g(\mcL_0,\dots,\mcL_{n-1},\mcA,\mcB)$,
$I_g(\mcL_0,\dots,\mcL_n,\mcH)$, and $I_g(\mcL_0,\dots,\mcL_n,s_n)$
are uniquely determined and can be constructed after base change from
$T$ to a Zariski cover of $T$.  Thus, in what follows, we perform such
base changes freely.

\mni
\textbf{(i)}
Denote $\mcA\otimes_{\OO_Y}\mcB$ by $\mcL_n$.
Denote by $\ol{J}$ a set of size $n+3$, $\{0,\dots,n-1,n,A,B\}$.
Denote by $\wt{J}$, resp. $J$, $J_n$, the subset $\{0,\dots,n-1,A,B\}$,
resp. $\{0,\dots,n\}$, 
$\{0,\dots, n-1\}$.  For every subset $J'\subset \ol{J}$, define
$J'_n$ to be the subset $J'\cap J_n$ of $J_n$.

\mni
Define $P$ to be the subset of the power set of $\ol{J}$ consisting of
A
subset $J'\subset \ol{J}$ is \emph{of type $P$} if for the subset
$J'_n=J'\cap J_n$, one of the following hold: 
$$
J'=J'_n,\ 
J'=J'_n\sqcup\{A\},\  J'=J'_n\sqcup\{B\},\
J'=J'_n\sqcup\{A,B\},\text{ or }
J'=J'_n\sqcup\{n\}.
$$
Define $P$ to be the collection of all subsets $J'\subset \ol{J}$ of
this type.  Said differently $J'$ fails to be in $P$ if and only if
$J'$ contains the subset $\{A,n\}$ or contains the subset $\{B,n\}$.
Notice, since $J'_n$ has size $\leq n$, every $J'$ in $P$ has size
$\leq n+2$.

\mni   
By Lemma \ref{lem-divisor3}, 
after base
change of $T$ to a Zariski cover, 
there exists an integer $m\geq 1$ 
and
sections $(t_0,\dots,t_{n-1},t_A,t_B)_{j\in J}$, 
$$
t_j:\OO(-m)\to \OO_Y,
$$ 
such that for every
subset $J'\subset J$, 
the restriction of $(t_j)_{j\in J'}$ to every fiber of
$p$ is a regular sequence.  
Up to replacing $\OO(1)$ by $\OO(m)$,
assume that $m$ equals $1$.  

\mni
For every $J'$ in $P$ of size $e'$, we next define a length-$e'$ 
sequence $t_{J'}$ of
sections of ample invertible sheaves
whose restriction to every fiber of $p$ is a regular
sequence.  Thus the zero scheme $Z_{J'}$ of this regular sequence is
flat over $T$.  If $e'\leq n+1$, by considering the intersection with
fibers of $g$,
every fiber of $Z_{J'}\to T$ is nonempty, and hence has pure dimension
$n+d-e'$.  If $e'$ equals $n+2$, then over every connected open
subcheme of $T$ where $Z_{J'}\to T$ has some nonempty fiber, then
every fiber is nonempty of pure dimension $d-2$, but there may well be
connected components of $T$ over which $Z_{J'}$ is empty.

\mni
For $J'\subset \wt{J}$, define $t_{J'}$ to be $(t_j)_{j\in J}$.  By
construction, this is a regular sequence on every fiber of $p$.
For every subset $J'_n\subset J_n$, whose size $e$ automatically
satisfies $e\leq n$,  
on every fiber of $p$, both of the sequences $(t_j)_{j\in
  J'}\sqcup(t_A)$ and $(t_j)_{j\in J'}\sqcup (t_B)$
are
regular.  Since
both the
images $\ol{t}_A$ and $\ol{t}_B$ are regular modulo the sequence
$(t_j)_{j\in J'}$, also $\ol{t_At_B}$ is regular modulo
$(t_j)_{j\in J'_n}$, i.e., $(t_j)_{j\in J'_n}\cup (t_At_B)$ is a
regular sequence on every fiber of $p$.  For the set
$J'=J'_n\sqcup\{n\}$,   
define $t_{J'}$
to be the sequence $(t_j)_{j\in J'_n} \cup (t_At_B)$ of
length $e+1$.  Define $Z_{J'}$ to be the zero scheme of
this sequence. 

\mni
By Lemma \ref{lem-divisor2} applied to the closed subschemes
$(Z_{J'})_{J'\in P}$ and the sequence of invertible
sheaves $(\mcL_0,\dots,\mcL_{n-1},\mcA,\mcB)$, 
there exists a sequence of positive integers
$(m_0,\dots,m_{n-1},m_A,m_B)$ and a sequence $(s_0,\dots,s_{n-1},s_A,s_B)$ 
of $\OO_Y$-module
homomorphisms, 
$$
s_j:\OO(-m_j)\to \mcL_j, \ s_A:\OO(-m_A)\to \mcA, \ s_B:\OO(-m_B)\to \mcB
$$
such that for every $J'$ in $P$ with size $e'$ 
and for every finite subset
$J''\subset J$ whose size $e''$ satisfies
$e''\leq n+d+1-e'$, the sequence
$(s_j)_{j\in J'}$ is regular when restricted to every fiber of $Z_{J'}\to
T$.  

\mni
Define $m_n=m_A+m_B$ and define $s_n$ to be $s_A\otimes s_B$,
$$
\OO(-m_A-m_B)\xrightarrow{s_A\otimes s_B} \mcA\otimes_{\OO_Y}\mcB
= \mcL_n.
$$
For every $J'$ in $P$ with size $e'$, for every
subset $J''_n\subset J_n$ of size $e''_n\leq n+d-e'$, both of the
following sequences are regular when restricted to fibers of
$Z_{J'}\to T$: $(s_j)_{j\in J'_n}\sqcup (s_A)$ and $(s_j)_{j\in
  J'_n}\sqcup (s_B)$.  Thus, also the sequence $(s_j)_{j\in
  J'_n}\sqcup (s_n)$ is regular.  

\mni
For every $i=0,\dots,n-1$, define $\mcL_i' = \OO(m_i)$ and
$\mcL_i''=\mcL_i(m_i)$.  Similarly, define $\mcA' = \OO(m_A)$,
$\mcA''=\mcA(m_A)$, $\mcB'=\OO(m_B)$, and $\mcB''=\mcB(m_B)$.  For
every $i = 0,\dots,n-1$, define $s'_i$ to be $t_i^{m_i}$, and define
$s_i''$ to be $s_i$.  Define $\alpha'$, resp. $\beta'$, 
to be $t_A^{m_A}$, resp. $t_B^{m_B}$.  Define $\alpha''$,
resp. $\beta''$, to be $s_A$, resp. $s_B$.

\mni
Let $(J',J'')$ be a partition of $J$.  If $n\in J'$, resp. if $n\in
J''$, by construction the sequence
$$
(t_j)_{j\in J'_n} \sqcup (s_j)_{j\in J''_n} \sqcup (t_A,t_B),
$$
respectively the sequence 
$$
(t_j)_{j\in J'_n} \sqcup (s_j)_{j\in J''_n} \sqcup (s_A,s_B),
$$ 
is regular on every fiber of
$p$.  Thus, also the sequence 
$$
(t_j^{m_j})_{j\in J'_n} \sqcup (s_j)_{j\in J''_n} \sqcup (t_A^{m_A},t_B^{m_B}),
$$
respectively the sequence
$$
(t_j^{m_j})_{j\in J'_n} \sqcup (s_j)_{j\in J''_n} \sqcup (s_A,s_B),
$$ 
is regular on every fiber of $p$.  Thus, the hypotheses of Corollary
\ref{cor-cones} are satisfied.  In particular, for every invertible
sheaf $\mc{H}$ on $Y$,
$$
\text{det}(Rg_*(\mc{H}\otimes \langle \mcL_{J',J''} \rangle)) \cong
I_g(\mcL_{J',J''}) \cong
$$  
$$
I_g(\mcL_{J',J'',A})\otimes_{\OO_C} I_g(\mcL_{J',J'',B}) \cong
\text{det}(Rg_*(\mc{H}\otimes \langle \mcL_{J',J'',A} \rangle)) \otimes_{\OO_C}
\text{det}(Rg_*(\mc{H}\otimes \langle \mcL_{J',J'',B} \rangle)).
$$

\mni
In the
K-group of locally free $\OO_Y$-modules, there is an identity,
$$
[\mcL_j]-[\OO_Y] = [\OO(-m_j)]\otimes\lt( ([\mcL_j'']-[\OO_Y]) -
([\OO(m_j)]-[\OO]) \rt) = 
$$
$$
[\OO(-m_j)]\otimes\lt( ([\mcL_j'']-[\OO_Y]) - ([\mcL'_j]-[\OO_Y]) \rt).
$$
Denote $m=m_0+\dots+m_n$.  
Via the multiadditivity of the operation $(\mcL_0,\dots,\mcL_n)\mapsto
\langle \mcL_0,\dots,\mcL_n
\rangle$, in the K-group there is an identity,
$$
\langle \mcL_0,\dots,\mcL_n \rangle = [\OO(m)]\otimes
\lt(\sum_{r=0}^n (-1)^r \sum_{(J',J''), \#J'' = r} \langle
  \mcL_{J',J''} \rangle \rt).
$$
Since $Rg_*$ and $\text{det}$ are additive, using Proposition
\ref{prop-cones}, it follows that $I_g(\mcL_1,\dots,\mcL_n)$ equals an
alternating tensor product of invertible sheaves
$I_g(\mcL_{J',J''})$.  Thus, by Corollary \ref{cor-cones}, there is an
isomorphism
$$
I_g(\mcL_0,\dots,\mcL_{n-1},\mc{A}\otimes_{\OO_Y}\mc{B}) \cong
I_g(\mcL_0,\dots,\mcL_{n-1},\mc{A})
\otimes_{\OO_C} 
I_g(\mcL_0,\dots,\mcL_{n-1},\mc{B}).
$$

\mni
\textbf{(ii)}
Via the same strategy as in the proof (i), this follows from the
corresponding statement in Proposition \ref{prop-Koszul}.

\mni
\textbf{(iii)}
Since $R(g\circ \iota)_*$ equals $Rg_*\circ R\iota_*$, it suffices to
prove an identity 
$$
\langle \mcL_0, \dots,\mcL_{n-1},\mcL_n \rangle =
[\mcH]\otimes R\iota_* \langle \iota^*\mcL_0,\dots,\iota^*\mcL_{n-1}
\rangle
$$ 
for an invertible $\OO_Y$-module $\mcH$.  Since $\iota^*$
is a ring homomorphism of K-rings, 
$$
\langle \iota^*\mcL_0, \dots, \iota^*\mcL_{n-1} \rangle 
\cong 
\iota^* \langle \mcL_0,\dots, \mcL_{n-1} \rangle.
$$
Thus, by the projection formula,
$$
R\iota_* \langle \iota^*\mcL_0, \dots, \iota^*\mcL_{n-1} \rangle
\cong
\langle \mcL_0, \dots, \mcL_{n-1} \rangle
\otimes 
[\iota_* \OO_Z].
$$
Finally, the resolution $s_n:\mcL_n^\vee \to \OO_Y$ of $\iota_*\OO_Z$
gives an identity,
$$
[\mcH]\otimes [\iota_*\OO_Z] = \lt( [\mcL_n]-[\OO_Y] \rt), 
$$
for $\mcH=\mcL_n^\vee$. 
\end{proof}

\mni
The projectivity hypothesis on $p$ is only necessary fpqc locally.

\begin{cor} \label{cor-intersect} \marpar{cor-intersect}
In the previous proposition, replace the hypothesis that there exists a
$p$-ample invertible sheaf on $C$ with the hypothesis that for some fpqc
morphism $T_0 \to T$ there exists a $T_0$-ample invertible sheaf on 
$C\times_T T_0$.  Also assume that $\pi$ has a Stein factorization.
Then the proposition still holds.
\end{cor}

\begin{proof}
Because of compatibility of pushforward and flat base change, the base
change of the Stein factorization,
$$
C\times_T T_0 \xrightarrow{\pi'\times \text{Id}} T'\times_T T_0
\xrightarrow{\rho \times \text{Id}} T_0,
$$
is a Stein factorization of $C\times_T T_0\to T_0$.  
Now use the same observation as in the previous proof:  because
$\pi'_*(\pi')^*\mcL$ equals $\mcL$ for every invertible sheaf on $T'$,
the invertible sheaves 
$I_g(\mcL_0,\dots,\mcL_{n-1},\mcA,\mcB)$,
$I_g(\mcL_0,\dots,\mcL_n,\mcH)$, and $I_g(\mcL_0,\dots,\mcL_n,s_n)$
are uniquely determined.  Thus, the invertible sheaves on $T'\times_T
T_0$ constructed using Proposition \ref{prop-intersect} satisfy the
fpqc descent condition.  
\end{proof}

\begin{cor} \label{cor-MG} \marpar{cor-MG}
Let $\pi:C\to T$ be a morphism that is proper, fppf of pure relative
dimension $d\geq 1$, and cohomologically flat in degree $0$.  Also
assume that every geometric fiber of $\pi$ satisfies Serre's condition
${\sS}_2$.  Then there exists a relative intersection sheaf for $\pi$ as in
Definition \ref{defn-intshf}.  Moreover, this intersection sheaf is
unique.  
\end{cor}

\begin{proof}
By the proposition, the intersection sheaves from Definition
\ref{defn-DP} satisfies the axioms from Definition \ref{defn-intshf}.
Moreover, the proof of the proposition proves that, via Axioms (ii) and (iii), 
the intersection
sheaf for an arbitrary $(n+1)$-tuple of invertible sheaves 
can be reconstructed from those $(n+1)$-tuples of invertible sheaves that
admit sections forming a regular sequence of length $n+1$.  For such
$(n+1)$-tuples, Axiom (iv) and induction on $n$ reduces to the case
that $n=0$.  For $n=0$, Axiom (v) uniquely determines the intersection
sheaf.  Thus, the relative intersection sheaf is unique.  
\end{proof}


\section{Tori and  Torsors} \label{sec-torsors}
\marpar{sec-torsors}

\mni
The intersection sheaves above are sufficient to construct Abel maps in
case the target has Picard rank one.  To deal with higher Picard rank,
it is necessary to generalize the intersection sheaf from
$\Gm{m}$-torsors to torsors for a more general group scheme $Q$ over $C$, as in
\cite[Expos\'{e} XVIII, Formulaire 1.3.8]{SGA4T3}.  
The group scheme $Q$ will be 
isomorphic to a product of copies of $\Gm{m}$, i.e., a split torus,
after pullback by a finite, \'{e}tale morphism $C_0\to C$.
Unfortunately, cohomological flatness in degree $0$ is not preserved
by finite, \'{e}tale morphisms.  

\begin{ex} \label{ex-etalecohfl} \marpar{ex-etalecohfl}
Let $k$ be a field.  Let $\ol{C}$ be a smooth, projective, geometrically
connected curve over $k$ of genus $g\geq 2$.  For simplicity assume
that $\ol{C}$ has a $k$-point, so that there exists an
invertible sheaf $\mcP$ on $\ol{C}\times_{\SP k} \Pic{2g-2}{\ol{C}/k}$
representing the relative Picard functor.  The pushforward
$\text{pr}_{2,*}\mcP$ is flat and of formation compatible with
arbitrary base change when restricted over the open complement of the
unique $k$-point parameterizing the dualizing sheaf $\omega$.  
On every open
neighborhood of $[\omega]$, this sheaf is neither
flat nor of formation compatible with arbitrary base change.

\mni
Let $T\subset
\Pic{2g-2}{\ol{C}/k}$ be a dense open subset.  On $\ol{C}\times_{\SP
  k} T$, define $\mcA$ to be the commutative, coherent sheaf of algebras 
$$
\mcA = \OO \oplus \mcP\epsilon,
$$
where $\epsilon^2$ equals $0$.  Define $\nu:C\to \ol{C}\times_{\SP k}
T$ 
to be the relative Spec,
$C=\SP \mcA$.  The projection $\pi:C\to T$ is projective, flat, and
even LCI.  Since $\pi_*\OO_C$ equals $\OO_T\oplus
\text{pr}_{T,*}\mcP$, $\pi$ is cohomologically flat in degree $0$ if
and only if $T$ does not contain the distinguished $k$-point
$[\omega]$.  
Now let
$\ol{C}_0 \to \ol{C}$ be the finite, \'{e}tale morphism of degree $2$
associated to a nontrivial $2$-torsion invertible sheaf $\mcT$ on
$\ol{C}$.  This extends uniquely to a finite, \'{e}tale morphism of
degree $2$, $C_0\to C$.  The morphism $C_0\to T$ is cohomologically
flat in degree $0$ if and only if $T$ contains neither $[\omega]$ nor
$[\omega\otimes_{\OO_{\ol{C}}} \mcT]$.  Thus, it can happen that $C\to
T$ is cohomologically flat in degree $0$, yet $C_0\to T$ is not
cohomologically flat in degree $0$.
\end{ex}

\begin{hyp} \label{hyp-etalecohfl} \marpar{hyp-etalecohfl}
Let $T$ be quasi-compact.
Let $\pi:C\to T$ be a morphism that is proper, fppf of pure relative
dimension $d\geq 1$, and whose geometric fibers $C_t$ are reduced, are
${\sS}_2$, and have only finite
geometric covers.
That last hypothesis says that for every \'{e}tale morphism $b:C_0\to
C_t$ satisfying the valuative criterion of properness (but $b$ need
not be quasi-compact) and with $C_0$ having only finitely many
connected components, then $b$ is a finite morphism.
These hypotheses holds if $C_t$ is
normal or if $C_t$ is an at-worst-nodal curve of ``compact type''.
\end{hyp}

\begin{defn} \label{defn-mult} \marpar{defn-mult}
A \emph{torus} over $C$ is a smooth group scheme $Q$ over $C$ that is
\'{e}tale locally isomorphic to $\Gm{m,C}^\rho$ for an integer
$\rho\geq 0$, the \emph{rank} of the torus.  The \emph{Cartier dual}
of $Q$ is the \'{e}tale group scheme $Q^D$ over $C$ whose associated
\'{e}tale sheaf of Abelian groups is the sheaf
$\text{Hom}_{C-\text{gr}}(Q,\Gm{m,C})$.  This sheaf is locally
constant with fiber $\ZZ^\rho$.  
\end{defn}

\mni
The group scheme $Q^D\to C$ is never quasi-compact if $\rho>0$.
However, since $C$ is quasi-compact, $Q^D$ is a countable increasing
union of open subschemes that are quasi-compact.  Because $C$ is
quasi-compact, and because of the
hypothesis on the geometric fibers of $\pi$, 
for the smallest open and closed
subscheme $Q_i^D$ of $Q^D$ containing a specified quasi-compact open, $Q_i^D$ is
finite over $C$.  
For every geometric point of $C$, some $Q^D_i$
contains a (finite) set of generators for the geometric fiber $Q^D$
(as a group).  Again using that $C$ is quasi-compact, there exists a
$Q^D_i$ that generates $Q^D$ as a group scheme.

\begin{ex} \label{ex-nonqc} \marpar{ex-nonqc}
If we drop the hypothesis on the fibers of $\pi$, this
can easily fail.  For instance, let $C$ be a nodal plane cubic, let
$b:C_0\to C$ be the unique finite, \'{e}tale morphism of degree $2$
with connected domain, and let $Q^D$ be the \'{e}tale group scheme
that is isomorphic to $\ZZ^2$ on each of the two irreducible
components of $C_0$, yet where the glueing isomorphisms at the two
nodes differ by an infinite order automorphism of $\ZZ^2$, e.g.,
$(m,n)\mapsto (m,m+n)$.  Presumably there is a hypothesis weaker than
unibranch that works and that allows the fibers to be nonreduced (yet
${\sS}_2$).  Reducedness of fibers is useful not only here, but also
because it implies cohomological flatness in
degree $0$ of $C_0\to T$ for all
finite, \'{e}tale covers of $C$.
\end{ex}

\mni
For 
every set $\Sigma$, denote by $L_{C,\Sigma}$ the \'{e}tale $C$-scheme
together with a set map $\lambda:\Sigma\to L_{C,\Sigma}(C)$ that
represents the functor associating to every \'{e}tale $C$-scheme $L$
the collection of all set maps $\text{Hom}_{\textbf{Set}}(\Sigma, L(C))$.  Thus
the sheaf of $L_{C,\Sigma}$ is locally constant with fiber $\Sigma$.   
locally $\lambda$ defines an isomorphism of the constant sheaf
In particular,
$L_{C,\ZZ^\rho}$ is the \'{e}tale group scheme over $C$ with fiber
$\ZZ^\rho$.  

\begin{lem} \label{lem-tilde}
There exists a finite, \'{e}tale, surjective morphism $C_{0}\to
C$ representing the functor that associates to every $C$-scheme
$S$ the set of all isomorphisms of \'{e}tale group $C$-schemes,
$\phi:L_{S,\ZZ^\rho}\to S\times_{C} Q^D$, mapping the basis elements of
$\ZZ^\rho$ into the subscheme $S\times_{C} Q^D_i$.  
\end{lem}

\begin{proof}
This is straightforward.  Since both $L_{C,[\rho]}$ and $Q^D_i$ are
finite, \'{e}tale over $C$, the Hom scheme
$H=\text{Hom}_{C}(L_{C,[\rho]},Q^D_i)$ is finite, \'{e}tale over
$C$.  The universal morphism over $H$ extends to a morphism of group
schemes $\phi:L_{H,\ZZ^\rho}\to H\times_{C} Q^D$.  The target is \'{e}tale
locally isomorphic to $L_{H,\ZZ^\rho}$, thus the determinant of $\phi$
is \'{e}tale locally well-defined up to a sign.  This determinant 
is an \'{e}tale
locally constant function.  Thus, there is an open and closed
subscheme $C_{0}$ of $H$ that is the locus on which this
determinant is $+1$ or $-1$.  As an open and closed subscheme of a
finite, \'{e}tale scheme over $C$, also $C_{0}$ is a finite,
\'{e}tale scheme over $C$.  Since $Q^D$ is \'{e}tale locally
isomorphic to $L_{C,\ZZ^\rho}$ and since $Q^D_i$ generates the group
scheme on all geometric fibers,
$C_{0} \to C$ is surjective.
\end{proof}

\begin{rmk} \label{rmk-tilde} \marpar{rmk-tilde}
By the lemma, $Q\times_C C_0$ is a split multiplicative group on
$C_0$.  By adjointness of pullback and restriction of scalars, there
is a natural morphism of group schemes $Q\to R_{C_0/C}(Q\times_C
C_0)$.  Checking \'{e}tale locally, this morphism is unramified and
locally split.  Thus, the lemma is just a global version of the
``standard'' argument that every torus embeds in the torus of a
``permutation module''.
\end{rmk}

\mni
For a morphism $a:C_0\to C$, define $C_1=C_0\times_C C_0$,
resp. $C_2 = C_0\times_C C_0 \times_C C_0$.   Denote the projections
by
$$
\pr{1},\pr{2}:C_1\to C_0,
$$
$$
\pr{1,2},\pr{1,3},\pr{2,3}:C_2\to C_1.
$$
A \emph{descent datum} of an affine $C$-scheme relative to $C_0$ is a
pair $(f_0:E_0\to C_0,\phi)$ of an affine morphism $f_0$ and an
isomorphism of $C_1$-schemes, $\phi:\pr{1}^*E_0\to \pr{2}^*E_0$ such
that the following \emph{cocycle condition} or \emph{descent
  condition} is satisfied,
$$
\pr{2,3}^*\phi \circ \pr{1,2}^*\phi = \pr{1,3}^*\phi.
$$
For descent data $(f_0:E_0\to C_0,\phi)$ and $(f'_0:E'_0\to
C_0,\phi')$, an \emph{morphism} between these descent data is an
morphism of $C_0$-schemes, $g_0:E_0\to E'_0$ such that
$\pr{2}^*g_0\circ \phi$ equals $\phi'\circ \pr{1}^*g_0$.  The identity is
a morphism, and the composition of two morphisms is an morphism.
Thus, descent data form a category.

\mni
For every affine morphism $f:E\to C$, the \emph{associated
  descent datum} is $E_0=E\times_C C_0$ with projection $f_0:E_0\to
C_0$ and with $\phi$ equal to the natural isomorphism 
$$
(E\times_{C,a} C_0)\times_{C_0,\pr{1}} C_1 \cong E\times_{C,a\circ
  \pr{1}} C_1 = E\times_{C,a\circ \pr{2}} C_1 \cong (E\times_{C,a}
C_0)\times_{C_0,\pr{2}} C_1.
$$
For an morphism of affine $C$-schemes, $g:E\to E'$, the associated
morphism of descent data is $g_0:E_0\to E_0'$,
$$
g\times \text{Id}_{C_0}: E\times_C C_0 \to E'\times_C C_0.
$$
Every descent datum isomorphic to the descent datum associated to an
affine $C$-scheme is an \emph{effective descent datum}.  The basic
result of effective fpqc descent is the following.

\begin{thm}
\cite[Th\'{e}or\`{e}me 2, p. 190-19]{FGA} 
\label{thm-descent} \marpar{thm-descent}
Assume that $C_0\to C$ is faithfully flat and quasi-compact.  For
every pair of affine $C$-schemes, $f:E\to C$ and $f':E'\to C$, every
morphism of the associated descent datum is associated to a unique
morphism $g:E\to E'$ of $C$-schemes.  Every descent datum of affine
schemes relative to $C_0$ is effective.  
\end{thm}

\mni
This is relevant here because $Q$-torsors are affine and fpqc.  Thus,
they can be used both as the affine scheme and the fpqc cover in the
previous theorem.  In particular, this leads quickly to an existence
result for ``induced torsors''.  Let $Q$ and $Q'$ be faithfully flat,
finitely presented group schemes over a scheme $C$.  Let $\psi:Q\to
Q'$ be a morphism of $C$-group schemes.  Denote by $m_{\psi}:Q\times_C
Q'\to Q'$ the left $Q$-action by multiplication on the \emph{right} by
the inverse, $m_{\psi}(q,q') = q' \psi(q)^{-1}$.  
For every left $Q$-torsor $E$,
say $m_E:Q\times_C E \to E$,
$C$, there is an induced ``diagonal'' left action of $Q$ on $Q'\times_C E$,
$$
m_{\psi,E}:Q\times_C Q'\times_C E \to Q'\times_C E, \ \pr{2}\circ
m_{\psi,E} = m_E\circ \pr{1,3}, \ \pr{2}\circ m_{\psi,E} =
m_{\psi}\circ \pr{1,2}.
$$
Altogether, this makes $Q'\times_C E$ into a $Q'\times_C Q$-torsor
over $C$. (This is the reason for using the inverse of the right
$Q$-action on $Q'$; so that it commutes with the left regular action
of $Q'$ on itself.)
The goal is to construct a left $Q'$-torsor $\psi_*E$ and a $Q$-invariant,
left $Q'$-equivariant morphism $p_{\psi,E}:Q'\times_C E \to \psi_*E$
realizing $Q'\times_C E$ as a $Q$-torsor over $\psi_*E$.  In
particular, $p_{\psi,E}$ is a categorical quotient of the diagonal
action.  Thus, if $p_{\psi,E}$ exists, then it is unique up to unique
isomorphism.  This strong uniqueness insures the cocycle condition for
a descent datum.

\begin{cor} \label{cor-psi} \marpar{cor-psi}
There exists a left $Q'$-torsor $\psi_*E$ and a morphism
$p_{\psi,E}:Q'\times_C E\to \psi_*E$ that is $Q'$-equivariant, that is
invariant for the diagonal $Q$-action, and that realizes $Q'\times_C
E$ as a $Q$-torsor over $\psi_*E$ for its diagonal $Q$-action.
\end{cor}

\begin{proof}
Since $Q$ is fpqc, and since $E$ is isomorphic to $Q$ as a $Q$-torsor
after passing to an fpqc cover of $C$, also
the projection morphism $a:E\to C$ is fpqc.  By Theorem
\ref{thm-descent},
\cite[Th\'{e}or\`{e}me 2, p. 190-19]{FGA},
it suffices to construct an $a$-descent datum for $\psi_*E$ and
$p_{\psi,E}$ satisfying the conditions.  Because of the strong
uniqueness, both the morphism $\phi$ in the descent datum and the
cocycle conditions are automatic.  Thus, it suffices to prove the
existence of a $Q'$-torsor $\psi_*E_a$ over $E$ 
and a morphism
$p_{\psi,E}:(Q'\times_C E)\times_{C,a}E \to \psi_*E_a$ 
from the $a$-pullback of
$Q'\times_C E$ to $\psi_*E_a$ that is left $Q'$-equivariant, that is
$Q$-invariant, and that realizes $(Q'\times_C E)\times_{C,a} E$ as a
$Q$-torsor over $\psi_*E_a$.

\mni
Since $E$ is a $Q$-torsor, the following morphism is an isomorphism,
$$
\lambda_E : Q\times_C E \to E\times_C E,\ \pr{1}\circ\lambda_E = m_E,
\  \pr{2}\circ\lambda_E =
\pr{2}. 
$$
Thus, there is an induced isomorphism,
$$
\text{Id}_{Q'} \times \lambda_E: Q'\times_C Q\times_C E \to Q'\times_C
E\times_C E.
$$
This is an isomorphism from the trivial $Q'\times_C Q$-torsor over $E$
to the pullback by $a$ of the $Q'\times_C Q$-torsor $Q'\times_C E$.
Consider the following morphism,
$$
r: Q'\times_C Q\times_C E \to Q'\times_C E, \ (q',q,e)\mapsto (q'\psi(q),e).
$$
This is left $Q'$-equivariant, it is $Q$-invariant, and it realizes
$Q'\times_C Q\times_C E$ as a $Q$-torsor over $Q'\times_C E$.  
Since $\text{Id}_{Q'}\times \lambda_E$ is an isomorphism,
there is a unique morphism
$$
p_{\psi,E,a}:Q'\times_C E\times_C E \to Q'\times_C E
$$
such that $p_{\psi,E,a}\circ (\text{Id}_{Q'}\times \lambda_E)$ equals
$r$.  Thus, $p_{\psi,E,a}$ is also left $Q'$-equivariant, it is
$Q$-invariant, and it realizes $Q'\times_C E\times_C E$ as a
$Q$-torsor over $Q'\times_C E$.   
\end{proof}

\mni
The operation $E\mapsto \psi_*E$ induces a $1$-morphism of classifying
stacks.  

\begin{defn} \label{defn-stack} \marpar{defn-stack}
For $C$ as in Hypothesis \ref{hyp-etalecohfl} and for $Q$ a torus over
$C$, 
denote by $BQ$ the $C$-stack classifying $Q$-torsors.  This is a
quasi-compact, quasi-separated algebraic $C$-stack with affine
diagonal.  Denote by $BQ_{C/T}$ the $T$-stack $\text{Hom}_T(C,BQ)$.
For every morphism $\psi:Q\to Q'$ of tori over $C$, denote by
$B\psi:BQ\to BQ'$, resp. by $B\psi_{C/T}:BQ_{C/T}\to BQ'_{C/T}$, 
the $1$-morphism defined by $E\mapsto \psi_* E$.
\end{defn}

\begin{prop} \label{prop-repcohfl} \marpar{prop-repcohfl}
If $\pi:C\to T$ is proper and fppf, then $BQ_{C/T}$ is an algebraic
$T$-stack that is locally finitely presented and whose diagonal is
quasi-compact and separated.  If $\pi$ satisfies
Hypothesis \ref{hyp-etalecohfl}, then $BQ_{C/T}$ has a coarse moduli
space $|BQ_{C/T}|$ that is locally finitely presented and
quasi-separated (typically not separated).  Assuming, moreover, that
$\pi$ has geometrically reduced fibers and that
the finite part $T'/T$ of the Stein factorization of $C/T$ is
\'{e}tale, fppf locally 
$BQ_{C/T}\to |BQ_{C/T}|$ is a gerbe for a commutative, fppf
group scheme that is the kernel of a morphism of tori.  Finally,
assuming further that a
$Q$-trivializing, finite, \'{e}tale, surjective cover 
$b:C_0\to C$ of $Q$ is geometrically integral over
the Stein factorization $T'$ of $C/T$, and assuming there exists a
$T'$-section of $C_0$,
then
this gerbe is split, i.e., there is a section $|BQ_{C/T}|\to BQ_{C/T}$.
\end{prop}

\begin{proof}
By  
\cite[Proposition 2.3.4]{Lieblich} and 
\cite[Theorem 1.5]{OHom},
$BQ_{C/T}$ is a locally finitely presented algebraic $T$-stack whose
diagonal is quasi-compact and separated.  The proof that the diagonal
is quasi-separated is roughly the same as the proof that $BQ_{C/T}$
has a coarse moduli space, so here it is quickly.
By Lemma \ref{lem-tilde}, there exists a finite, \'{e}tale morphism
$a:C_0\to C$ such that $a^*Q$ is isomorphic to $\Gm{m,C_0}^\rho$.
Quasi-separatedness of this stack is equivalent to quasi-compactness
of the Isom scheme $\text{Isom}_{Q/T}(E,E')$ for $Q$-torsors $E$ and
$E'$ over $C$, not only for $Q$-torsors on $C/T$, but also for
$Q$-torsors on every base change
$C\times_T S /S$.   
By fpqc descent, Theorem \ref{thm-descent}, \cite[Th\'{e}or\`{e}me 2,
p. 190-19]{FGA},   
$Q$-torsors on $C$ are equivalent to 
$C_0/C$ descent data of $Q$-torsors.  On $C_0$, the $Q$-torsors are
equivalent to $\rho$-tuples of invertible sheaves.  So the Isom scheme
is an $\rho$-fold fiber product over $T$ of the Isom schemes of the
component invertible sheaves.
These Isom schemes are quasi-compact by
\cite[Corollaire 7.7.8]{EGA3}. 
The scheme
parameterizing the isomorphism $\phi$ on $C_1$ is again described in
terms of Isom schemes of invertible sheaves on $C_1$.  This Isom
scheme is again quasi-compact by 
\cite[Corollaire 7.7.8]{EGA3}. 
Of course to descend, this data must satisfy the cocycle
condition.  These are closed conditions: the subscheme parameteizing
data satisfying the cocycle condition is a closed subscheme.  A closed
subscheme of a quasi-compact scheme is again quasi-compact.  Thus,
altogether, the stack is quasi-separated.

\mni
In the same way, the coarse moduli functor of $Q$-torsors over $C$
is relatively a finitely presented, 
affine scheme over the coarse moduli functor of
$Q$-torsors on $C_0$.  Since the pullback of $Q$ to $C_0$ is
$\Gm{m,C_0}^\rho$, this functor is a quasi-separated, locally finitely
presented algebraic $T$-space by Hypothesis
\ref{hyp-etalecohfl} and 
\cite[Theorem 7.3]{Artin}.

\mni
Next, assume that $\pi$ has geometrically reduced fibers and that
$T'/T$ is \'{e}tale.  Then for every \'{e}tale cover $C_0\to C$, the
finite part of the Stein factorization, $T_0\to T$, 
is \'{e}tale.  In particular,
after a finite, \'{e}tale base change of $T$, assume that $T'$ is a
disjoint union of $m$ copies of $T$.  After a further fppf base change,
essentially by $C\to T$, assume that there is an $m$-tuple
$s=(s_1,\dots,s_m)$ of sections
$s_i:T\to C$ splitting the morphism $C\to T'$.  After a
further finite, \'{e}tale base change of $T$, there exist finite, \'{e}tale
morphism 
$b:C_0\to C$ and a trivialization $\psi: Q_0
\to
\Gm{m,C_0}^\rho$,
$Q_0 =Q\times_C C_0$,
such that $C_0\to T'$ is the Stein factorization and
such that there exists a lift,
$s_{i,0}:T\to C_0$, of each section $s_i$.  

\mni
Thus, assume now that $C_0\to T'$ has geometrically integral fibers,
and assume that
there exists a section $s_0:T'\to C_0$ of the
Stein factorization.  This induces a section $s=b\circ s_0:T'\to C$.
Since $C_0\to T'$ is
projective and flat with integral geometric fibers, every automorphism
of a $\Gm{m,C_0}^\rho$-torsor $E$ is just multiplication by a $T'$-section
of $\Gm{m,T}^\rho$.  
Denote by
$BQ_{C/T,s}$ 
the rigidified stack 
of
$Q$-torsors $E$ together with a trivialization $\tau:s^*E\cong s^*Q$ as
$s^*Q$-torsors.  By Theorem \ref{thm-descent}, 
\cite[Th\'{e}or\`{e}me 2, p. 190-19]{FGA},  
every automorphism of such a pair $(E,\tau)$ is uniquely determined by
the induced
automorphism of the corresponding object of $(BQ_0)_{C_0/T,s_0}$.
As noted above, these automorphisms are trivial.  Thus, the stack
$BQ_{C/T,s}$ is an algebraic space.  There is a forgetful morphism
$\Phi:BQ_{C/T,s} \to BQ_{C/T}$.  This is a torsor for the pullback to
$BQ_{C/T}$ of the group scheme $s^*Q$:  any two trivializations $\tau$
differ by post-composing by multiplication by a section of $s^*Q$.
Thus, the induced map of coarse moduli spaces is an isomorphism, i.e.,
$BQ_{C/T,s}$ is equivalent, as a $T$-stack and thus also as a
$T$-algebraic space, to the coarse moduli space $|BQ_{C/T}|$.
Therefore, the morphism $BQ_{C/T}\to |BQ_{C/T}|$ admits a section.

\mni
There is a commutative 
group scheme $R$ over $BQ_{C/T,s}$ defined as the pushforward
via $\pi$ of the Isom scheme of the torsor $E$.  For the finite,
\'{e}tale 
covers
$C_0\to C$, resp. $C_1\to C$, there are associated group schemes $R^0$
and $R^1$.  Again using Theorem \ref{thm-descent}, 
\cite[Th\'{e}or\`{e}me 2, p. 190-19]{FGA}, there is an exact sequence
of commutative group schemes,
$$
0 \to R \to R^0 \stackrel{d^0}{\twoheadrightarrow} R^1,
$$
where $d^0$ sends $g_0$ to $\phi'\circ \pr{1}^*g_0 - \pr{2}^*g_0\circ \phi.$
By construction, both $R^0$ and $R^1$ are tori: $R^0$ is isomorphic to
the Weil restriction  of
$\Gm{m,T'}^\rho$ 
for the finite, \'{e}tale cover $T'/T$, and $R^1$ is the Weil
restriction of a split torus for the finite, \'{e}tale cover
arising from its Stein factorization over $T$.  Thus, \'{e}tale
locally on $T$, the morphism $d^0$ is a morphism between split tori.
Choosing splittings of these tori, $d^0$ is equivalent to an integer
valued matrix.  Considering the Smith normal form of this matrix, $R$
is equivalent to a product of copies of $\Gm{m,T}$ and copies of
finite, flat group scheme $\mu_\ell$ for various integers $\ell$
(possibly divisible by the characteristic, thus not necessarily
\'{e}tale).  The $1$-morphism $BQ_{C/T}\to BQ_{C/T,s}$ is a gerbe for this
group scheme $R$.
\end{proof}


\section{Intersection Torsors} \label{sec-inttor}
\marpar{sec-inttor}

\mni
There is a category $\textbf{Torus}_C$ 
whose objects are pairs $(a:C_0\to C,Q_0)$
of a finite, \'{e}tale morphism $a$ and a torus $Q_0$ on $C_0$.  For
objects $(a':C'_0\to C,Q'_0)$ and $(a:C_0\to C,Q_0)$, a morphism from
the first to the second is a pair $(b:C'_0\to C_0,\psi)$ of a
$C$-morphism $b$ (automatically finite and \'{e}tale) and a morphism of
group schemes over $C'_0$, $\psi:Q_0 \times_{C_0} C'_0\to Q'_0$ (note,
this is contravariant in the second argument).
Composition and identities are defined in the evident manner.  This
category is fibered over the category of finite, \'{e}tale
$C$-schemes, and each fiber category is an additive category (it is not
an Abelian category, because there are no cokernels satisfying the
usual axioms for an Abelian category).  The clivage associates to
$(a:C_0\to C,Q_0)$ and $b:C'_0\to C_0$ the object $(a':C'_0\to
C_0,Q_0\times_{C_0} C'_0)$. 

\mni
For a $Q_0$-torsor $E_0$ on $C_0$, there
is an associated torsor $b^*E_0$ on $C'_0$ for $Q_0\times_{C_0}
C'_0$.  By Corollary \ref{cor-psi}, 
there is an associated $Q'_0$-torsor $\psi_* b^*E_0$ on
$C'_0$. 
This \emph{pullback} torsor $(b,\psi)^*E_0$ is
contravariant in $(b,\psi)$ and covariant in $E_0$.  

\begin{defn} \label{defn-Qintshf} \marpar{defn-Qintshf}
Let $\pi:C\to T$ be a morphism satisfying Hypothesis
\ref{hyp-etalecohfl}.  For every object $(a:C_0\to C,Q_0)$ of
$\textbf{Torus}_C$, 
an \emph{intersection datum} for $\pi$ is a tuple
$$
(n,T_0\to T,p_0:Y_0\to T,g_0:Y_0\to T_0 \times_T C_0, 
(\mcL_0,\dots,\mcL_{n-1},E))
$$
of an integer $n\geq 0$, a $T$-scheme $T_0$, a proper, fppf morphism
$p_0$ of relative dimension $\leq d+n$ that is projective fppf locally
over $T_0$, a proper, perfect morphism of $T_0$-schemes, an $n$-tuple
of invertible sheaves $(\mcL_0,\dots,\mcL_{n-1})$ on $Y_0$, and a
torsor over $Y_0$ for $Q_0\times_{C_0} Y_0$.  For fixed
$(n,T_0,p_0,g_0)$, an \emph{isomorphism} between $(n+1)$-tuples
$(\mcL_0,\dots,\mcL_{n-1},E)$ and $(\mcL'_0,\dots,\mcL'_{n-1},E')$ is
an $n$-tuple of isomorphisms between the component invertible sheaves
$\mcL_i \to \mcL'_i$, and an isomorphism of torsors $E\to E'$ under
$Q_0\times_{C_0} Y_0$.  For a datum as above, and for a morphism in
$\mcC$, $(b,\psi):(C'_0,Q'_0)\to (C_0,Q_0)$, the
$(b,\psi)$-\emph{pullback} of the datum is the datum
$$
(n, T_0\to T, p_0\times \text{Id} : Y_0\times_{C_0} C'_0 \to T,
g_0\times \text{Id} : Y_0\times_{C_0} C'_0 \to T_0 \times_T C'_0,
$$
$$
(\pr{Y_0}^*\mcL_0, \dots,\pr{Y_0}^*\mcL_{n-1},\psi_*\pr{Y_0}^* E)).
$$
Similarly, for a datum as above and for a morphism of $T$-schemes,
$T_1\to T_0$, the \emph{pullback datum} is
$$
(n,T_1,Y_0\times_{T_0} T_1 \to T_1, Y_0\times_{T_0} T_1
\xrightarrow{g_0\times \text{Id}} C_0\times_{T_0} T_1,
(\pr{Y_0}^*\mcL_0, \dots, \pr{Y_0}^*\mcL_{n-1},\pr{Y_0}^*E)).
$$

\mni
A \emph{relative intersection torsor} for $\pi$ is an assignment to
every intersection datum of a section
$I_{g_0}(\mcL_0,\dots,\mcL_{n-1},E]$ over $T_0$ of $|(BQ_0)_{C_0/T}|$
that satisfies all of the following axioms.
\begin{enumerate}
\item[(0)] The assignment is invariant under isomorphism of
  $(n+1)$-tuples.
\item[(i)] The assignment is $\mf{S}_n$ -invariant in
  $(\mcL_0,\dots,\mcL_{n-1})$.  
\item[(ii)] The assignment is additive in every argument of
  $(\mcL_0,\dots,\mcL_{n-1},E)$.  
\item[(iii)] The assignment is compatible both for
  $(b,\psi)$-pullbacks and for pullbacks of $T$-schemes.
\item[(iv)] For every $i=0,\dots,n-1$, for every section
  $s_i:\mcL_i^\vee\ to \OO_{Y_0}$ that is regular on every fiber of
  $p_0$, for the associated Cartier divisor
  $\iota_0:Z_0\hookrightarrow Y_0$, 
$$
I_{g_0\circ
    \iota_0}(\iota_0^*\mcL_0,\dots,
  \iota_0^*\mcL_{i-1},\iota_0^*\mcL_{i+1},\dots,
  \iota_0^*\mcL_{n-1},\iota^*E] = I_{g_0}(\mcL_0,\dots,\mcL_{n-1},E].
$$
\item[(v)] If $n$ equals $0$, and if $Q_0$ equals $\Gm{m,C_0}$, then
for the $\Gm{m,C_0}$-torsor $E$ of an invertible sheaf $\mcL_0$,
$I_{g_0}(E]$ equals
$\text{det}(Rg_*(\mcL_0))\otimes_{\OO_{C_0}}\text{det}(Rg_*\OO_{Y_0})^\vee$. 
\end{enumerate}
\end{defn}

\begin{cor} \label{cor-MG2} \marpar{cor-MG2}
Let $\pi:C\to T$ be a morphism satisfying Hypothesis
\ref{hyp-etalecohfl}.  Also assume that every geometric fiber of $\pi$
satisfies Serre's condition ${\sS}_2$.  Then there exists a relative
intersection torsor for $\pi$, and this intersection torsor is unique.
\end{cor}

\begin{proof}
This is a consequence of Corollary \ref{cor-MG} and fpqc descent,
Theorem \ref{thm-descent}, 
\cite[Th\'{e}or\`{e}me 2, p. 190-19]{FGA}. 
Corollary \ref{cor-MG} gives
a relative intersection torsor whenever $Q_0$ is the split torus
$\Gm{m,C_0}^\rho$: for a $\Gm{m,Y_0}^\rho$-torsor
$E=(E_1,\dots,E_\rho)$ with $E_j$ associated to an invertible sheaf
$\mcN_j$, the intersection $\Gm{m,C_0}^\rho$-torsor
$(I_1,\dots,I_\rho)$ has $I_j$ associated to the invertible sheaf
$I_{g_0}(\mcL_0,\dots,\mcL_{n-1},\mcN_j)$.  Additivity in the last
argument implies that this is compatible with $(b,\psi)$-pullbacks for
morphisms $\psi:\Gm{m,C'_0}^\rho \to \Gm{m,C'_0}^{\rho'}$.    

\mni
Having constructed the relative intersection torsor whenever $Q_0$ is
a split torus, now Lemma \ref{lem-tilde} and fpqc descent, Theorem
\ref{thm-descent}, 
\cite[Th\'{e}or\`{e}me 2, p. 190-19]{FGA},
extends the relative intersection
torsor to general torsors $Q_0$.  The cocycle condition for fpqc
descent is precisely compatibility with $(b,\psi)$-pullbacks.
\end{proof}


\section{Degrees} \label{sec-degrees}
\marpar{sec-degrees}

\mni
The last aspect of Abel maps has to do with degrees of torsors.  For
this, we make a strong hypothesis that is satisfied in the case of
Abel maps.

\begin{hyp} \label{hyp-310} \marpar{hyp-310}
Let $T$ be an excellent, Noetherian scheme.  Let $\pi:C\to T$ be a
projective, smooth morphism of pure relative dimension $d$. 
\end{hyp}

\begin{defn} \label{defn-Lambda} \marpar{defn-Lambda}
For every finite, flat morphism $b:C_0\to C$ with
$\pi_0=\pi\circ b$ smooth, for every torus $Q_0$ on $C_0$, define
$\Lb(Q_0/C_0/T)$ to be the pushforward with respect to $\pi_0$ of
the cocharacter lattice $\text{Hom}_{C_0-\text{gr}}(\Gm{m,C_0},Q_0)$
(as an \'{e}tale sheaf).  This is contravariant in $C_0$, it is
covariant in $Q_0$, and it is compatible with \'{e}tale base change of
$T$.
\end{defn}

\mni
The \'{e}tale group scheme $\Lb(Q/C/T)$ is the natural target for
the degree map.  In particular, when $T$ is Spec of a finite field,
this degree is a logarithmic height, and the map to $\Lb(Q/C/T)$
is a ``multiheight'' as in \cite{BGS}, \cite{Gubler}.

\begin{defn} \label{defn-degree} \marpar{defn-degree}
A \emph{degree datum} for $\pi$ is a tuple
$$
(b:C_0\to C, Q_0, (\mcL_1,\dots,\mcL_{d-1},E))
$$
of a finite, flat morphism $b$ such that $\pi_0=\pi\circ
b$ is smooth, a torus $Q_0$ on $C_0$, a $(d-1)$-tuple of invertible
sheaves $(\mcL_1,\dots,\mcL_{d-1})$ on $C_0$, and a $Q_0$-torsor $E$
on $C_0$.  Isomorphisms and pullbacks are as in Definition
\ref{defn-Qintshf}. A \emph{relative degree map} for $\pi$ is an
assignment to every degree datum of a section
$\text{deg}_{\pi_0}(\mcL_1,\dots,\mcL_{d-1},E)$ of
$\Lb(Q_0/C_0/T)$ over $T_0$ that satisfies all of the following
axioms.   
\begin{enumerate}
\item[(0)] The assignment is invariant under isomorphism.
\item[(i)] The assignment is $\mf{S}_{d-1}$-invariant in
  $(\mcL_1,\dots,\mcL_{d-1})$.  
\item[(ii)] The assignment is additive in every argument of
  $(\mcL_1,\dots,\mcL_{d-1},E)$.  
\item[(iii)] The assignment is compatible with base change of $T$.
\item[(iv)] For every finite, flat morphism $g:C_1\to C_0$
  such that $\pi_1=\pi_0\circ g$ is smooth, for every $(d-1)$-tuple
  $(\mcL_0,\dots,\mcL_{d-1})$ of invertible sheaves on $C_0$, for
  every $Q_0\times_{C_0} C_1$-torsor $F$ on $C_1$, 
$$
\text{deg}_{\pi_1}(g^*\mcL_1,\dots,g^*\mcL_{d-1},F) \text{ equals }
g^*\text{deg}_{\pi_0}(\mcL_1,\dots, \mcL_{d-1},I_g(F]).
$$
\item[(v)] For every morphism of tori $\psi:Q_0\to Q'_0$,
$$
\text{deg}_{\pi_0}(\mcL_1,\dots,\mcL_{d-1},\psi_*E) \text{ equals }
\psi_*\text{deg}_{\pi_0}(\mcL_1,\dots, \mcL_{d-1},E).
$$
\item[(vi)] When $\text{pr}_0$ has connected geometric fibers, and
when $Q_0$ equals $\Gm{m,C_0}$, for the associated
  invertible sheaf $\mcL_d$ of $E$,
  $\text{deg}_{\pi_0}(\mcL_1,\dots,\mcL_{d-1},E)$ equals the virtual
  rank of $R\pi_{0,*}(\langle \mcL_1,\dots,\mcL_{d-1},\mcL_d \rangle)$.
\end{enumerate}
\end{defn}

\begin{prop} \label{prop-degree} \marpar{prop-degree}
Under Hypothesis \ref{hyp-310}, there exists a relative degree map.
The relative degree map is unique.  Denoting by $\Lb(Q/C/T)$ also the
\'{e}tale group scheme over $T$ whose \'{e}tale sheaf is as above,
the relative degree map defines a morphism of $T$-schemes,
$\Pic{}{C/T}\times_T \dots \times_T \Pic{}{C/T} \times_T |BQ_{C/T}|
\to \Lb(Q/C/T)$.
\end{prop}

\begin{proof}
This is proved by the same method as in the proof of Corollary
\ref{cor-MG2}.  The definition of the degree in the 
general case reduces to the case that
$Q_0$ equals $\Gm{m,C}$.  
The main point is that for every element $a\in \gamma^{d+1}$, the
virtual rank of
$R\pi_*(a)$ is zero.  Since the virtual rank is locally constant, this
can be checked after base change to the residue field of a point of
$T$.  Now the claim follows by Riemann-Roch, cf. \cite[Corollary 8.10,
Proposition 16.12]{ManinK}.  
Modulo $\gamma^{d+1}$, the rule
$(\mcL_1,\dots,\mcL_{d-1},\mcL_d) \mapsto \langle
\mcL_1,\dots,\mcL_{d-1},\mcL_d\rangle$ is multiadditive.  Since
$R\pi_*(-)$ is also additive for the K-group, it follows that
$\text{deg}_\pi(\mcL_1,\dots,\mcL_{d-1},\mcL_d)$ is multiadditive.

\mni
For every $(d-1)$-tuple $(\mcL_1,\dots,\mcL_{d-1})$, 
$\langle \mcL_1,\dots,\mcL_{d-1} \rangle$ is in
$\gamma^{d-1}$.  For every virtual object $F$ in the K-ring having virtual
rank $r$ and determinant $\text{det}(F)$, the class $F-r[\OO_C]$ is congruent to
$\text{det}(F)-[\OO_C]$ modulo $\gamma^2$.  Since the gamma filtration
is multiplicative, the virtual rank of
$R\pi_*(-)$ on $\langle \mcL_1,\dots,\mcL_{d-1}\rangle \otimes (F-r[\OO_C])$
is the same as $\text{deg}_\pi(\mcL_1,\dots,\mcL_{d-1},\text{det}(F)
)$.  In particular, for $g:C_1\to C$ as above, for
an invertible sheaf $\mcM$ on $C_1$, the virtual object 
$g_*[\mcM] - g_*[\OO_{C_1}]$ has
virtual rank $0$ and determinant equal to $I_g(\mcM]$.  Thus, via the
projection formula,
$$
\text{deg}_{\pi\circ g}(g^*\mcL_1,\dots,g^*\mcL_{d-1},\mcM) =
R\pi_*(Rg_*(\langle g^*\mcL_1,\dots,g^*\mcL_{d-1}\rangle\otimes
([\mcM]-[\OO_{C_0}]))) = $$
$$
R\pi_*(\langle \mcL_1,\dots,
\mcL_{d-1}\rangle\otimes (g_*[\mcM]-g_*[\OO_{C_1}])) = \text{deg}_{\pi}(\mcL_1,\dots,\mcL_{d-1},I_g(\mcM]).
$$   
\end{proof}

\mni
Since $\pi$ is smooth, the geometric fibers are reduced, and the same
holds for all finite, \'{e}tale covers.  Thus, for every finite,
\'{e}tale morphism $C_0\to C$, the morphism $C_0\to T$ is
cohomologically flat in degree $0$.  Since the geometric fibers of
$\pi$ are smooth curves, they satisfy Serre's condition ${\sS}_2$.  Thus,
the hypothesis above implies Hypothesis \ref{hyp-etalecohfl}.

\begin{defn} \label{defn-curvemaps} \marpar{defn-curvemaps}
For every algebraic stack $\mcX$ over $C$, 
define $\textit{CurveMaps}(\mcX/T)$ to be the $T$-stack whose objects are
commutative diagrams
$$
\begin{CD}
Y_0 @> g_0 >> C \\
@V p VV @VV \pi V \\
T_0 @>> h > T
\end{CD}
$$
together with a $1$-morphism over $C$, $\zeta:Y_0 \to \mcX$.
Here $h:T_0\to T$ is a morphism of schemes, and $p$ is a flat, proper
morphism of relative dimension $\leq 1$.  When $\mcX$ is the
classifying stack $BQ$, the $T$-stack above is denoted
$\textit{CurveMaps}(BQ/T)$.  
\end{defn}

\mni
By Lemma \ref{lem-310},
$g_0:Y_0 \to T_0\times_T C$ is perfect and proper.  Every flat, proper
morphism of relative dimension $\leq 1$ is projective fppf locally
over $T_0$, cf. \cite[Proposition 3.3]{dJHS}.  Thus, by Corollary
\ref{cor-MG2}, there is an intersection torsor $I_{g_0}(E]$ that is a
section of $|BQ_{C/T}|$ over $T_0$.

\begin{prop} \label{prop-310} \marpar{prop-310}
The $T_0$-section $I_{g_0}(E]$ of $|BQ_{C/T}|$ defines a $1$-morphism
$\textit{CurveMaps}(BQ/T)\to |BQ_{C/T}|.$  When the diagram above is a
family of stable sections over $C$, this $1$-morphism agrees with the
$1$-morphisms constructed in \cite{dJHS} and \cite{Zhu}.
\end{prop}

\begin{proof}
By the axioms, the intersection torsor
is functorial for pullback in $T_0$ and
for isomorphism of $Q$-torsors over $Y_0$.  Thus, this is a
$1$-morphism.  The \emph{construction} of an intersection torsor via
det and Div is precisely how the Abel maps are constructed 
in \cite{dJHS} and \cite{Zhu}.
\end{proof}

\begin{hyp} \label{hyp-char0} \marpar{hyp-char0}
Let $T$ be an excellent, Noetherian scheme.  Let
$\pi:C\to T$ be a projective, smooth morphism of pure relative
dimension $1$, i.e., a family of projective, smooth curves over $T$
(not necessarily geometrically connected).
\end{hyp}

\mni
Denote by $T'\to T$ the finite part of the Stein factorization of
$\pi$.  Since $\pi$ is smooth, $T'/T$ is finite and \'{e}tale.
Since $\pi$ has relative dimension $1$, the relative degree map
defines a morpism of $T$-schemes,
$$
\text{deg}_{Q/C/T}: |BQ_{C/T}| \to \Lb(Q/C/T).
$$

\begin{prop} \label{prop-degree2} \marpar{prop-degree2}
Assuming Hypothesis \ref{hyp-char0}, resp. assuming both Hypothesis
\ref{hyp-char0} and that $\OO_T(T)$ contains a characteristic
$0$ field, 
the stack $BQ_{C/T}$ is smooth and
the coarse moduli space $|BQ_{C/T}|$ is fppf, resp. $BQ_{C/T}$ is
smooth and
the coarse moduli
space is quasi-compact and smooth.
Also the degree morphism is fppf, resp. the morphism is surjective and
smooth.  

\mni
For each field $K$ and $K$-valued point 
$e\in \Lb(Q/C/T)(K)$, for $b:C_0\to C$ as in Lemma \ref{lem-tilde},
for every $K$-point $E$ of
$(BQ_0)_{C_0/T}$ with degree $b^*e\in \Lb(Q_0/C_0/T)(K)$, the norm
$I_b(E]$ is a $K$-point of $\Lb(Q/C/T)$ with degree $e$.  Thus
$\text{deg}_{Q/C/T}(K)$ is surjective if $C_0$ has a $K$-point.

\mni
The degree morphism is a torsor for the kernel 
$|BQ_{C/T}^\tau|$.  Locally on $T$ for the fppf topology,
the kernel group scheme 
admits a
finite morphism to an Abelian scheme.  Thus $|BQ_{C/T}^\tau|$ is a
proper, fppf group scheme, resp. it is
an extension of a (commutative) finite,
\'{e}tale 
group scheme over $T$ by an Abelian scheme over $T$.  
\end{prop}

\begin{proof}
Since $\Lb(Q/C/T)$ is \'{e}tale over $T$, to prove that
$\text{deg}_{Q/C/T}$ is flat and finitely presented, resp. smooth,
it is equivalent to prove that
$|BQ_{C/T}|$ is flat and finitely presented over $T$, resp. smooth
over $T$.  This can be checked fppf locally on $T$.  By Proposition
\ref{prop-repcohfl}, after fppf base change of $T$, the stack
$BQ_{C/T}$ is a gerbe over $|BQ_{C/T}|$ for a commutative, fppf group
scheme $R$ that is the kernel of a morphism of tori.  In
characteristic $0$, $R$ is a smooth group scheme.  Thus, to prove that
$|BQ_{C/T}|$ is fppf, resp. to prove that $|BQ_{C/T}|$ is smooth in
characteristic $0$, it suffices to prove that $BQ_{C/T}$ is smooth
over $T$.  

\mni
The infinitesimal deformation theory of $BQ_{C/T}$, and
more general ``restriction of scalars'' stacks, is given in
\cite[Section 5.7]{OHom}.  The pullback by the identity section of the
relative sheaf of differentials of $Q/C$ is a locall free sheaf
$\mf{q}^\dagger$ over $C$.  The cotangent complex of $BQ\to C$ is
isomorphic to the pullback of the complex concentrated in
cohomological degree $+1$, i.e., 
$\mf{q}^\dagger[-1]$.  This is a pullback because the group scheme $Q$ is
commutative; for a non-commutative group scheme $G$, twist by the universal
$G$-torsor and the adjoint action on the sheaf of left-invariant
differential $\mf{g}^\dagger$.  For an Artinian scheme $\SP A'$, for
a quotient by a square-zero ideal 
$$
0 \to I \to A'\to A \to 0,
$$ 
for a morphism $\SP A\to
T$, for a $Q$-torsor $E_A$ over $C_A=C\times_T \SP A$, the obstruction
to deforming $E_A$ is an element in $H^2(C_A,\mf{q}\otimes_A I)$, and
this vanishes because $C_A$ has dimension $1$.  Thus, $BQ_{C/T}$ is
smooth over $T$.

\mni
Therefore the kernel $|BQ^\tau_{C/T}|$ of the degree morphism is an
fppf group scheme, resp. a smooth group scheme in characteristic $0$.
Because the degree morphism is flat and finitely presented,
resp. smooth in characteristic $0$, this morphism of group schemes is
a torsor under the kernel assuming that the morphism is surjective.

\mni
To prove surjectivity in the general case, it suffices to prove
surjectivity in case $T$ equals $\SP K$ for a field $K$.  To prove
surjectivity of points (not surjectivity of the induced map of
$K$-points), it suffices to check after base change to a larger field.
Thus, consider the special case that there exists a finite, \'{e}tale,
surjective morphism $b:\Cc_0\to \Cc$ and an
isomorphism $\psi:Q_0\to \Gm{m,\Cc_{0}}^\rho$.  Assume, moreover, that
for the Stein factorization $T'=\SP \OO_C(C)$, also $\Cc_0$ is
geometrically integral over $T'$.  Finally, assume that there exists a
$T'$-point $s_0$ of $\Cc_0$.
By Lemma \ref{lem-tilde} this holds after a finite field extension of
$K$.
By Proposition \ref{prop-repcohfl}, under these hypotheses
the stacks are split gerbes over the coarse moduli spaces, so they
have the same (isomorphism classes of) $K$-points.
For every $K$-point $e$ of $\Lb(Q/C/T)$ and for every $K$-point $E_0$ of
$\Lb(Q_0/C_0/T)$ of degree $b^*e$, by Definition
\ref{defn-degree}(iv) and Proposition \ref{prop-degree}, the norm
$E=I_b(E_0]$ has degree $e$.  Thus, to prove that $\text{deg}_{Q/C/T}$
is (geometrically) surjective, it suffices to prove the same for
$\text{deg}_{Q_0/C_0/T}$.  
Via $\psi$, the degree morphism is 
$$
|(B\Gm{m,C_0})_{C_0/T}^\rho| \to \Lb(\Gm{m,C_0}^\rho/C_0/T).
$$
This is the $T'/T$-Weil restriction of the associated morphism of
$T'$-schemes, 
$$
(\Pic{}{C_0/T'})^\rho \to \ZZ^\rho.
$$
Since the degree of $\OO_{C_0}(\ul{s_0})$ equals $1$, this degree
morphism is surjective, and even $\text{deg}_{Q_0/C_0/T}(T)$ is
surjective.  Thus also $\text{deg}_{Q/C/T}(T)$ is surjective on
$T$-points.  Since this holds for all sufficiently large field
extensions, 
the degree morphism is surjective on points (on geometric points
always, and on rational points assuming that $\Cc_0$ has a
$T'$-point).
Note, moreover, that
the kernel $(\Pic{0}{\Cc_{0}/T'})^\rho$ of the degree morphism
is an Abelian scheme over $T'$.  Since $T'\to T$ is finite and
\'{e}tale,
the $T'/T$-Weil restriction is also
an Abelian scheme over $T$.  

\mni
Denote by $R$ the quotient torus, $R = b_*Q_0/Q$, for the unramified morphism
$Q\to b_*Q_0$ of tori
on $C$.   
By the same proof as in Proposition \ref{prop-repcohfl}, $\pi_*R$ is an
fppf group scheme on $T$ that is, fppf locally, a product of a torus
and a commutative, finite, flat group scheme.  The quotient $M$ by
$\pi_*b_*Q_0 \cong \Gm{m,T}^\rho$ is also such an fppf group scheme,
cf. Corollary \ref{cor-psi}.
By the long exact sequence of fppf cohomology, there is an exact
sequence of fppf group schemes,
$$
0 \to M \to |BQ_{C/T}| \xrightarrow{b^*} 
(\Pic{}{\Cc_0/T})^\rho \xrightarrow{u} |BR_{C/T}|.
$$
The scheme $M$ is affine, and every locally closed subscheme of
$(\Pic{}{\Cc_0/T})^\rho$ is separated.  Thus, $|BQ_{C/T}|$ is
separated.  Then the same argument also implies that $|BR_{C/T}|$ is
separated.  So the kernel of $u$ is a closed subgroup scheme.  
Altogether, $|BQ_{C/T}|$ is an $M$-torsor over the
closed image of $b^*$.  

\mni
Since $M$ is a product of a torus and a finite, flat
group scheme, in order to prove that $M$ is a finite, flat group scheme, 
it suffices to prove that $M$ has finite order $m$.
Denote by $m$ the degree of the finite, \'{e}tale morphism $b$.
For a torsor $E$ on $\Cc_K$, 
$I_b(b^*E)$ is multiplication of $E$ by $m$.  Since $M$ equals the
kernel of $b^*$, $M$ has order $m$.
Therefore $M$ is a commutative, finite, flat group scheme that is
\'{e}tale locally 
a product of group schemes $\mu_\ell$ over $T$.

\mni
Also, since
$\Lb(Q_0/\Cc_{0}/T)$ is torsion-free, the degree of $b^*E$
equals $0$ if and only if the degree of $E$ equals $0$.  Thus, the
image under $b^*$ of $|BQ^\tau_{\Cc/T}|$ is the intersection of the
closed image of $b^*$ with $|(BQ_0)^\tau_{\Cc_{0}/T}| =
(\Pic{0}{\Cc_{K0}/T})^\rho$.  
This intersection is a 
closed subgroup scheme of an
Abelian scheme over $T$.  Altogether, $|BQ^\tau_{\Cc/T}|$ is a
commutative, 
proper group scheme over $T$ that is fppf.  In characteristic
$0$, such a group scheme is a direct product of an Abelian scheme and
a commut
is an extension of a finite, \'{e}tale group scheme by
an Abelian variety.  In characteristic $0$, such a group scheme is an
extension of a 
finite, \'{e}tale group scheme by  
an Abelian scheme (the Abelian scheme is the connected component of
the identity).  Of course over geometric points of $T$, this extension
is split by Chevalley's structure theorem.
\end{proof}

\begin{prop} \label{prop-degree3} \marpar{prop-degree3} 
With hypotheses as in Proposition \ref{prop-degree2}, assume that $T$
equals $\SP K$ for $K$ a field, and assume that $C\to T'$ is a geometrically
integral curve of genus $g(C)\geq 2$.  The exponent of the torsion group
$\Lb(Q/C/T)(K)/|BQ_{C/T}|(K)$ divides the degree $[K_0:K']$ of
every residue field $K_0$ of every closed point for every finite,
\'{e}tale, surjective morphism $b:C_0\to C$ such that $b^*Q$ is
isomorphic to $\Gm{m,C_0}^\rho$.
\end{prop}

\begin{proof}
For each such field extension $K_0/K$, replacing $C_0$ be the
irreducible component that contains the $K_0$-point, then $C_0$ is a
geometrically integral $K_0$-curve with a $K_0$-point.  Thus, by
Proposition \ref{prop-degree2}, $\text{deg}_{Q/C/K}(K_0)$ is
surjective, i.e., $\text{deg}_{Q_{K_0}/C_{K_0}/K_0}(K_0)$ is
surjective.  Since $\text{deg}_{Q_{K_0}/C_{K_0}/K}$ is the
$K_0/K$-Weil restriction of $\text{deg}_{Q_{K_0}/C_{K_0}/K_0}$, for
every $e\in \Lb(Q/C/K)(K)$, the multiple $[K_0:K]e$ is the Weil
restriction of $e\otimes_K K_0$.  Thus, for $E_0$ a $K_0$-point of
$BQ_{C/K}$ whose $Q_{K_0}/C_{K_0}/K_0$-degree equals $e\otimes_K K_0$,
the $Q/C/K$-degree of the norm $I_{K_0/K}(E_0)$ equals $[K_0:K]e$.  So
the order of $\ol{e}$ in the cokernel of $\text{deg}_{Q/C/K}(K)$
divides $[K_0:K]$.
\end{proof}

\mni
Finally, in order to make sense of the asymptotics of quantities
depending on the degree, there is a notion of a ``positive
structure'', i.e., a monoid of
positive degrees.

\begin{defn} \label{defn-pos} \marpar{defn-pos}
A \emph{positive structure}
on $\Lb(Q/C/T)$ is an open and closed
subscheme $\Lb^+(Q/C/T)$ that is a submonoid scheme, that
generates a finite index subgroup scheme of $\Lb(Q/C/T)$,
and that is \emph{sharp}, i.e., it 
contains no subgroup scheme other than the
trivial group scheme.  The positive structure is \emph{saturated},
resp. \emph{finitely generated}, if its fibers over geometric points
of $T$ are saturated (for multiplication by positive integers),
resp. finitely generated (as a commutative monoid).
\end{defn}

\begin{lem} \label{lem-sat} \marpar{lem-sat}
Every saturated positive structure generates $\Lb(Q/C/T)$.  
\end{lem}

\begin{proof}
This can be checked after base change to geometric points of $T$.  
By hypothesis, for every $e\in \Lb(Q/C/T)$,
there exists an integer $m>0$ such that $m\cdot e$ is contained in the
subgroup generated by $\Lb^+(Q/C/T)$.  Every nonnegative linear
combination of elements of $\Lb^+(Q/C/T)$ is an element of
$\Lb^+(Q/C/T)$.  Thus there exists $e_+,e_- \in \Lb^+(Q/C/T)$
such that $m\cdot e$ equals $e_+-e_-$.  Thus $m\cdot e + e_-$ is in
$\Lb^+(Q/C/T)$.  Since $m\geq 0$ and since $e_-$ is in
$\Lb^+(Q/C/T)$, 
also $(m-1)\cdot e_-$ is in $\Lb^+(Q/C/T)$.  Thus $m\cdot e + e_-
+ (m-1)\cdot e_-$ is in $\Lb^+(Q/C/T)$, i.e., $m\cdot (e+e_-)$ is
in $\Lb^+(Q/C/T)$.  Since this submonoid is saturated, $e+e_-$ is
in $\Lb^+(Q/C/T)$.  Therefore $e=(e+e_-) - e_-$ is in the subgroup
generated by $\Lb^+(Q/C/T)$.  
\end{proof}

\begin{lem} \label{lem-pos} \marpar{lem-pos}
Let $b:C_0\to C$ and $\psi:Q_0\to \Gm{m,C_0}^\rho$ 
be a trivializing finite, \'{e}tale cover as above.  Assume, moreover,
that $b$ is Galois with Galois group $\Gamma$.  For every
$\Gamma$-stabilized positive structure on $\Lb(Q_0/C_0/T)$, the
intersection of $\Lb^+(Q_0/C_0/T)$ with $\Lb(Q/C/T)$ is a
positive structure $\Lb^+(Q/C/T)$ 
on $\Lb(Q/C/T)$.  If $\Lb^+(Q_0/C_0/T)$ is saturated,
resp. finitely generated, then also $\Lb^+(Q/C/T)$ is saturated,
resp. finitely generated.
\end{lem}

\begin{proof}
The intersection $\Lb^+(Q/C/T)$ is an open and closed subscheme that is a
subsemigroup scheme and contains only the trivial group scheme.  It
remains to check that $\Lb^+(Q/C/T)$ generates a finite index
subgroup scheme.  This can be checked on geometric fibers over $T$.
Denote by $m$ the index in $\Lb(Q_0/C_0/T)$ of the subgroup generated
by $\Lb^+(Q/C/T)$.  Denote by $n$ the order of $\Gamma$.  For
every geometric point $e\in \Lb(Q/C/T)$, by hypothesis, the
element $m\cdot e$ in $\Lb(Q_0/C_0/T)$ is a $\ZZ$-linear combination of
elements of $\Lb^+(Q_0/C_0/T)$, say
$$
m\cdot e = \sum_i c_i f_i, \ \ c_i\in \ZZ,\ \  f_i\in \Lb^+(Q_0/C_0/T).
$$
Then also
$$
(n\cdot m)\cdot e = \sum_{\gamma\in \Gamma} \gamma\bullet(m\cdot e) =
\sum_i c_i (\sum_{\gamma \in \Gamma} \gamma\bullet f_i).
$$
By hypothesis, each $\gamma\bullet f_i$ is in $\Lb^+(Q_0/C_0/T)$.
Thus the sum $\sum_i \gamma\bullet f_i$ is an element of
$\Lb^+(Q_0/C_0/T)$ that is
$\Gamma$-invariant.  By descent, this is an element in
$\Lb(Q/C/T)$, thus it is in $\Lb^+(Q/C/T)$.  So the subgroup
generated by $\Lb^+(Q/C/T)$ has finite index dividing $n\cdot m$.

\mni
The intersection of a saturated submonoid of an Abelian group with a
subgroup is also saturated.  Thus, 
if $\Lb^+(Q_0/C_0/T)$ is saturated, then $\Lb^+(Q/C/T)$ is
also saturated.  Over geometric points of $T$, $\Lb(Q_0/C_0/T)$ is
isomorphic to $\ZZ^\rho$.  Assume that $\Lb^+(Q_0/C_0/T)$ is a
finitely generated monoid.
The intersection of any finite collection
of finitely generated monoids in $\ZZ^\rho$ is a finitely generated
monoid (the problem of algorithmically finding a finite set of
generators is known as the ``vertex enumeration problem'').  In
particular, the intersection of the finitely generated monoid
$\Lb^+(Q_0/C_0/T)$ and the finitely generated monoid (even
subgroup) $\Lb(Q/C/T)$ is a finitely generated monoid
$\Lb^+(Q_0/C_0/T)$.  
\end{proof}

\section{Acknowledgments} \label{sec-ack}
\marpar{sec-ack}

\mni
I am grateful
to Aise Johan de Jong and Eduardo Esteves who each 
pointed out that an Abel map
as in \cite{dJHS} must exist under weaker hypotheses. 
I thank Yi Zhu for many
discussions about Abel maps through the years.
I thank Chenyang Xu for many discussions; this work grew from our
joint project on rational points over global function fields for
rationally simply connected varieties.
I was supported
by NSF Grants DMS-0846972 and DMS-1405709, as well as a Simons
Foundation Fellowship. 

\bibliography{my}
\bibliographystyle{alpha}

\end{document}